\documentclass[11pt]{amsart}
\usepackage[margin=1.4in]{geometry}
\usepackage{amsmath,amsfonts,amssymb,amsthm,mathtools}
\usepackage{cases}
\usepackage{url}
\usepackage[all]{xy}
\usepackage{tikz-cd}
\usepackage{array}
\usepackage{mathtools}
\usepackage{faktor}
\usepackage{caption}
\usepackage{verbatim,eucal}
\usepackage{xfrac}
\usepackage{MnSymbol}
\usepackage{multicol}

\usepackage{letltxmacro}
\LetLtxMacro\Oldfootnote\footnote

\usepackage{hyperref}% order of hyperref and showkeys matters,% hyperref should be last package loaded but cleverref after it
\hypersetup{colorlinks=true,citecolor=blue!70!black,linkcolor=red!60!black,linktocpage=true}% change the way hyperrefs appear
\usepackage[capitalise,nameinlink]{cleveref}
\crefname{enumi}{part}{parts}% Refer to items as "Part"
\makeatother%something for copying to beamer slides?
\numberwithin{equation}{section}

\newtheorem{thm}[equation]{Theorem} 
\newtheorem{prop}[equation]{Proposition}
\crefname{prop}{Proposition}{Propositions}

\newtheorem{lemma}[equation]{Lemma} 
\newtheorem{cor}[equation]{Corollary}
\crefname{cor}{Corollary}{Corollaries}

\newtheorem{example}[equation]{Example}
\newtheorem{remark}[equation]{Remark}

\DeclareMathOperator{\Ext}{Ext}

\DeclareMathOperator{\im}{Im}

\DeclareMathOperator{\Span}{Span}

\DeclareMathOperator{\characteristic}{char}
\newcommand{\GL}{{\rm{GL}}}
\newcommand{\cchar}{{\rm{char\, }}}

\newcommand{\FF}{\mathbb F}
\newcommand{\RR}{\mathbb R}
\newcommand{\DOT}{\setlength{\unitlength}{1pt}\begin{picture}(2.5,2)
          (1,1)\put(2,3.5){\circle*{3}}\end{picture}}

\newcommand{\Hom}{\mbox{\rm Hom\,}}

\renewcommand{\ker}{\mbox{\rm Ker\,}}

\newcommand{\ot}{\otimes}

\newcommand{\cH}{\mathcal{H}}
\newcommand{\CC}{\mathbb{C}}

\newcommand{\del}{\partial}
\DeclareMathOperator{\codim}{codim}

\newcommand{\HH}{{\rm HH}}
\newcommand{\ZZZ}{{\rm Z}} %COCYCLES
\newcommand{\BBB}{{\rm B}} %COBOUNDARIES
\newcommand{\CCC}{{\rm C}} %COCHAINS
\newcommand{\Wedge}{
{\textstyle \bigwedge}}

\newenvironment{smallpmatrix}
  {\left(\begin{smallmatrix}}
  {\end{smallmatrix}\right)}

\begin{document}

%\DisableFootNotes
%\nocite{*}% for checking if references were cited, combine with \usepackage{checkreferences},
%but gives error for every \cref

\begin{abstract}
We examine
the Hochschild cohomology
governing graded deformations 
for 
finite cyclic groups acting on polynomial
rings.
We classify the infinitesimal graded deformations
of the skew group algebra
$S\rtimes G$ for a cyclic group $G$
acting on a polynomial ring $S$.
This gives all graded deformations
of the first order. We are particularly interested in the case
when the characteristic of the underlying
field divides the order of the acting group,
which complicates the determination of 
cohomology.
\end{abstract}

\title[Deformation cohomology]
{Deformation cohomology
for cyclic groups\\ acting on polynomial rings
}

\date{January 2024}

\author{Colin M.\ Lawson}
\address{Department of Mathematics\\
Stephen F.\ Austin State University,
Nacogdoches, Texas 75965, USA}
\email{colin.lawson@sfasu.edu}

\author{Anne V.\ Shepler}
\address{Department of Mathematics, University of North Texas,
Denton, Texas 76203, USA}
\email{ashepler@unt.edu}


\thanks{The second author was partially supported by Simons grants 
429539 and 949953.}

\maketitle

\section{Introduction}
Hochschild cohomology governs
deformations of algebras:  Every deformation
arises from a Hochschild $2$-cocycle,
but the converse is false,
and obstructions to
lifting a $2$-cocycle
to a deformation are witnessed by the Gerstenhaber bracket
(a Lie bracket) on Hochschild cohomology.
The $2$-cocycles are often
called {\em infinitesimal
deformations}
or {\em deformations of the first order}.
When $A$ is a graded algebra,
the Hochschild cohomology $\HH^{\DOT}(A)$
inherits the grading,
and the graded deformations
of $A$ 
all arise from 
infinitesimal deformations
of graded degree $-1$
(see \cite{GS86}).
Thus a first step
in determining the graded deformations
of a given algebra $A$
centers on describing 
the space $\HH_{-1}^2(A)$
of
{\em infinitesimal graded deformations}, i.e., the space of
{\em first-order graded deformations}
of $A$.
%
%.
The groups $\HH_{i}^j(A)$
are more generally
called
{\em graded Hochschild cohomology groups},
see \cite{Gordon08}.

We consider
a finite group $G\subset \GL(V)$
acting on a finite-dimensional
vector space $V\cong \FF^n$ over a field $\FF$
and take the induced
action on the polynomial ring $S(V)$,
the symmetric algebra of $V$.
Deformations of the natural
semidirect product algebra
$S(V)\rtimes G$
include Lusztig's
graded Hecke algebras
(see \cite{Lusztig88,Lusztig89}),
rational Cherednik algebras
and
symplectic reflection algebras
(see \cite{Gordon08} and \cite{EG}),
and Drinfeld Hecke algebras more
generally (see \cite{Drinfeld}).
The Hochschild cohomology of 
$A=S(V) \rtimes G$
has been described
in the {\em nonmodular setting},
when $\cchar \FF$ and 
the group order $|G|$ are coprime,
for $\FF$ algebraically
closed (see \cref{NonModularCohDeg2} below).
Much less is known in the {\em modular setting},
when $\cchar \FF$ divides $|G|$, as the group ring $\FF G$
is no longer semisimple and carries its
own nontrivial cohomology
(e.g., see \cite{CS97} and \cite{SiegelWitherspoon}).
Here we investigate the
case when the acting group $G$
is cyclic and we assume
$\cchar \FF\neq 2$.

%%%%%%%%%%%%%%%%%%%%%%%
\begin{thm}
\label{MasterThm}
Let $G\subset \GL(V)$ be a finite cyclic group acting on  $V \cong \FF^n$.
The space of 
infinitesimal
graded deformations of $A=S(V)\rtimes G$
is isomorphic as an $\FF$-vector space to
$$
\begin{aligned}
\HH^2_{-1}(A)
\ \cong\ 
(V^G/\im T)^*
\oplus
\big( 
V \ot\, \Wedge^2 V^*
\big)^G
\, \oplus\hspace{-2ex}
\bigoplus_{\substack{h \in G \\ \codim V^h =1 }}
\hspace{-3ex}
\left(\FF \oplus
\big(V/V_h \ot (V^h)^*\big)
\right)^{\chi_h}
\, \oplus\hspace{-2ex}
\bigoplus_{\substack{h\in G \\ \codim V^h =2 }}
\hspace{-3ex}
\left(
V/V_h
\right)^{\chi_h}
.
\end{aligned}
$$
\end{thm}
%%%
%%%
\noindent
Here, $\chi_h$ is an analogue of the Hochschild character,
see \cref{LinearCharacter},
and $T:V\rightarrow V$
is the {\em transfer map}
$T=\sum_{h\in G} h$ on $V$.

We recover the description of
$\HH^2_{-1}(A)$
in the nonmodular setting for $\FF$  algebraically closed, 
see \cite{SW08}
and \cref{NonModularCohDeg2}
below.
In that setting, 
{\em bireflections} (with fixed point spaces of codimension $2$)
along with the identity $1_G$
contribute
to this space of infinitesimals,
but 
{\em reflections} in the group do not.
Over $\RR$ or $\CC$,
graded deformations like the symplectic reflection algebras
and rational 
Cherednik algebras 
(see \cite{Gordon08}
and \cite{EG})
have 
parameters only supported on bireflections
and $1_G$ for this reason.

In the modular setting, 
we find
{reflections} 
in the group $G$ 
also contributing to the space of
infinitesimal graded deformations.
We expect this since examples
(see 
 \cite{KrawzikShepler}
and \cite{SW18})
of
graded deformations
like Drinfeld Hecke algebras
arise from parameters supported
on reflections in the group, 
not just bireflections
and $1_G$.  In these graded deformations of
$A=S(V)\rtimes G$, the reflections
record some deforming
of the semidirect product
structure
whereas the bireflections
record some deforming of 
the commutativity of $S(V)$.
The key complication in the modular setting
centers on the frequent
lack of a decomposition
$V=V^h\oplus (V^h)^\perp$ 
preserved by the
centralizer $Z(h)$ of $h$ in $G$, even for
$G$ abelian
(see \cref{ApplicationSection}).
We obtain more information
in the modular setting when
a Jordan canonical form
is available, see
\cref{PartsNonmodular}.

Marcos, Mart\'inez-Villa, and  Martins
\cite{MMM04, MMM04Addendum}
(see also Cibils and Redondo \cite{CibilsRedondo05})
examine
the Hochschild cohomology
of semidirect products $S\rtimes G$
for an $\FF$-algebra $S$
upon which a finite group $G$ acts by automorphisms
and describe actions under which
$$
\HH^i(S\rtimes G)
\cong
\Big(\mathbin{\coprod}_{(g,h)\in G\times G}
\Ext^i_{S^e}(Sg, Sh) 
\Big)^{G\times G}.
$$
To describe the cohomology
explicitly,
we take a different approach here.
For a cyclic group $G$, 
we directly
construct a twisted product
resolution 
(see \cite{SW19}) of $A=S(V)\rtimes G$
that combines
a convenient periodic
resolution of 
$\FF G$ with the Koszul resolution 
of $S(V)$.  This approach allows
us to give a concrete
description of the cohomology, which
can then be used to determine deformations
that simultaneously
generalize graded Hecke algebras
and universal enveloping algebras,
see \cref{ApplicationSection},
\cite{Norton}, and \cite{SW12}.
Also see Negron~\cite{Negron15}
and Briggs and Witherspoon \cite{BW22}
for related ideas.
%

%%%%%%%%%%%%%%%%%%%%%%
\subsection*{Outline of paper}
%%%%%%%%%%%%%%%%%%%%%%
%%%%%%%%%%%%%%%%%%%%%%
We recall some basic
facts about skew group algebras
and Hochschild cohomology in
\cref{background}
and review
a formulation of 
the
cohomology $\HH^{\DOT}(S(V) \rtimes G)$ in the nonmodular setting in
\cref{NonmodularSection}.
In \cref{PeriodicResolutionSection}, we 
use a periodic
resolution for a
cyclic group $G$ to construct
a twisted product resolution for $A=S(V)\rtimes G$ using the Koszul resolution for $S(V)$.
We decompose 
the space $\HH_{-1}^2(A)$
of
infinitesimal graded deformations
into contributions from
each group element
in \cref{DecompositionSection}.
We give cocycle conditions 
in \cref{CodimSection}
in terms of the dimension of
fixed point spaces.
Unique cocycle representatives
giving cohomology
are identified in \cref{UniqueRepsSection},
and we use these representatives
in \cref{MainTheoremSection}
to give the cohomology explicitly
as a vector space.
Lastly, in \cref{ApplicationSection},
we demonstrate how these results may be used
to find deformations
by considering the transvection
groups acting on $2$-dimensional
vector spaces.

%%%%%%%%%%%%%%%%%%%%%%%%%%%
%%%%%%%%%%%%%%%%%%%%%%%%%%
%%%%%%%%%%%%%%%%%%%%%%%%%%
%%%%%%%%%%%%%%%%%%%%%%%%%%
%%%%%%%%%%%%%%%%%%%%%%%%%%
%%%%%%%%%%%%%%%%%%%%%%%%%%%
%%%%%%%%%%%%%%%%%%%%%%%%%%

\section{Hochschild cohomology
and skew group algebras}
\label{background}

We take a finite group $G\subset \GL(V)$
acting on $V\cong \FF^n$ and consider
the induced action of $G$ on the polynomial ring
$S(V)$.
We take all tensor
products over the field $\FF$,
$\ot = \ot_{\FF}$, 
and assume $\cchar \FF \neq 2$ throughout.  We assume
all algebras are associative
$\FF$-algebras.

\subsection*{Group actions}
For any $\FF G$-module $M$,
we write $\,^g m$ for the image 
of $m$ in $M$ under the action of 
$g$ in $G$ to distinguish
from the product of $g$ and $m$
in any algebra containing both.
We take the usual induced action
on functions $f:M\rightarrow M'$
defined by
$(\, ^g f)(m)=\, ^g(f(\, ^{g^{-1}}m))$ for $g$ in $G$ and $m$ in $M$,
for any $\FF G$-module $M'$,
and we always take the trivial $G$ action on $\FF$.
%%%%%%%%%%%%%%%%%%%%%%%%%%%%%%%%%%%%%%%%%%
%%%%%%%%%%%%%%%%%%%%%%%%%%%%%%%%%%%%%%%%%%%5
\subsection*{Transfer map}
We use the classical {\em transfer map} 
in modular
invariant theory 
(see \cite{CW11})
for a group $G\subset \GL(V)$ restricted to the vector space $V$: Define
\begin{equation}
    \label{TransferMap}
T:V\longrightarrow V,
\quad
v\longmapsto \sum_{h\in G} \ ^hv
\, .
\end{equation}

%%%%%%%%%%%%%%%
%%%%%%%%%%%%%%%
\subsection*{Invariant subspaces
and characters} 
We denote the $G$-invariants
in any $\FF G$-module $M$ by
$M^G =
\{m \in M 
:
\,^gm = m
\ 
\text{for all $g \in G$}
\, \}$
and, more generally,
denote
the $\chi$-invariants
in $M$ by
$$
M^\chi=\{m\in M: \, ^gm=\chi(g)\, m
\text{ for all } g \in G\}
\, 
$$
for any linear character $\chi:G\rightarrow \FF^{\times}$
of $G$.
Specifically for $M=V$
and $h$ in $G$, we set  
$$
V^h=\ker(1-h)=\{v \in V \,:\, \,^hv = v \}
\quad\text{ and }\quad
V_h=\im(1-h)=\{v- \,^hv \,:\, v\in V\}
\, .
$$
When $G$ is abelian,
$G$ fixes set-wise
both $V^h$ and $V_h$ for any $h$ in $G$
and
we define
a linear character
giving the determinant of
$G$ acting on $V/V^h$, an analogue of the 
{\em Hochschild character:} 
\begin{equation}
   \label{LinearCharacter}
\chi_h:G\longrightarrow \FF^{\times},
\quad
\chi_h(g) := 
\det
\left[
\begin{matrix}
g
\end{matrix}
\right]_{V/V^h}
\,.
\end{equation}

%%%%%%%%%%%%%%%
%%%%%%%%%%%%%%%%
\subsection*{Skew group algebras}
Recall that the 
{\em skew group algebra} 
$S(V) \rtimes G=S\# G$
is the natural semidirect product algebra:
$S(V) \rtimes G=S(V) \ot \FF G$ as an $\FF$-vector space with multiplication given by 
$$
(s \ot g) \cdot (s' \ot g')
=
s(\,^gs') \ot gg'
\quad
\text{for all $s,s' \in S(V)$ and $g,g' \in G$}
\,.
$$
Note that we identify $1_G$ and $1_\FF$
in $\FF G$
and identify 
$\FF G$
with $\FF\ot \FF G$
and $S(V)$ with
$S(V)\ot \FF$, subspaces of
$S(V)\ot \FF G$.
%%%%%%%%%%%%%%%
%%%%%%%%%%%%%%%%%%%%%%%%%%%%%%%
%%%%%%%%%%%%%%%%

%%%%%%%%%%%%%%%
%%%%%%%%%%%%%%%%v
\subsection*{Identifications} 
For any $\FF G$-module $M$,
we identify spaces under the
$\FF G$-module isomorphism
\begin{equation}
\label{identifications}
\Hom_{\FF}(\Wedge^jV, M)
\cong
M \ot \Wedge^jV^*
\quad
\text{for each $j \ge 0$}
\,
\end{equation}
for dual space $V^*=\Hom_{\FF}(V,\FF)$
so that
$
\left(\Hom_{\FF}(\Wedge^jV, M)\right)^G
=(M \ot \Wedge^{j}V^*)^G
$.
Any bimodule over an $\FF$-algebra $A$
is a left $A^e$-module for $A^e = A \ot A^{op}$, the enveloping algebra of $A$,
with $A^{op}$ the opposite algebra of $A$,
and we
also
identify 
the spaces
$\Hom_{\! A^e}(A\ot M\ot A, A)$
and
$\Hom_{\FF}(M, A)$.

%%%%%%%%%%%%%%%
%%%%%%%%%%%%%%%%
%%%%%%%%%%%%%%%
%%%%%%%%%%%%%%%%

%%%%%%%%%%%%%%%
%%%%%%%%%%%%%%%%

\subsection*{Gradings}
The group $G$ acts on $S(V)$
by graded automorphisms
when we take the
natural grading on $S(V)$ 
by polynomial degree
with generators forming a vector space
basis of $V$ in degree $1$.
This grading
induces a grading on 
$A = S(V) \rtimes G$ 
after we set the
degree of each group element to zero.

%%%%%%%%%%%%%%%
%%%%%%%%%%%%%%%%v
%%%%%%%%%%%%%%%
%%%%%%%%%%%%%%%%

%%%%%%%%%%%%%%%
%%%%%%%%%%%%%%%%
\subsection*{Hochschild cohomology}
For an $\FF$-algebra $A$, the Hochschild cohomology of $A$ is 
$$
\HH^{\DOT}(A)
:=
\HH^{\DOT}(A, A) 
=
\Ext_{A^e}^{\DOT}(A,A)
\, .
$$
The cohomology
$\HH^{\DOT}(A)$ may be computed
as the homology of the complex
that arises from
applying
$\text{Hom}_{A^e}(\text{---}, A)$
to the bar resolution
of $A$ with $m$-th term
$A\ot A^{\ot m}\ot A$
for $m\geq 0$,
see \cite{W19}
and \cite{Redondo14} for example.

\subsection*{Grading
on cohomology}
If $A$ is a graded algebra, then
$\HH^m(A)$ inherits the induced grading
on the bar resolution \cite{W19},
and we denote the homogeneous
component of degree $i$ by
$\HH^m_{i}(A)$.
Specifically, we identify
$
\Hom_{\! A^e}(A\ot A^{m}\ot A, A)$ and $\Hom_{\FF}(A^{\ot m}, A)
$
and take the usual grading on
$A^{\ot m}$ with
$\deg(a_1\ot \cdots \ot a_m)
=\sum_j \deg(a_j)$
for $a_j$ homogeneous
in $A$,
so that
$\gamma$ in $\Hom_{\! A^e}(A\ot A^{m}\ot A, A)$ has degree $i$
when
$\deg (\gamma(a'\ot a \ot a''))
=i+\deg(a)$ for all $a\in A^{\ot m}$ and
$a', a''\in A$.

%%%%%%%%%%%%%%%
%%%%%%%%%%%%%%%%

\subsection*{Graded deformations}
Recall that an
$\FF$-algebra $A_t$
is a graded deformation of a graded algebra $A$ over $\FF[t]$ 
if $A_t$ is a graded algebra
over  $\FF[t]$
for $\deg t=1$ 
with $A_t\cong A[t]$
as an $\FF[t]$-vector space
and $A_t/tA_t \cong A$
as an $\FF$-algebra
(see \cite{GS86}
and \cite{BG96}).
We may
identify a graded deformation
$A_t$ with
$\FF[t]\ot_\FF A$
with multiplication
given by
\begin{equation}\label{multiplication}
a*b=ab + \mu_1(a \ot b)t
+ \mu_2(a \ot b)t^2+ \ldots
\quad\text{ for } a,b\in A
\end{equation}
extended to be linear over $\FF[t]$
for some $\FF$-linear maps
$\mu_i: A\ot A \rightarrow A$
homogeneous of degree $-i$.
(Note this forces the sum in \cref{multiplication} to be finite
as $\mu_i(a\ot b)=0$
for large enough $i$.)
The first multiplication
map $\mu_1$ is 
then necessarily a Hochschild $2$-cocycle of $A$ of degree $-1$,
called the {\em infinitesimal}
of $A_t$.
Thus to classify all graded
deformations of $A$,
we are interested in first
determining the space
$$
\HH^2_{-1}(A) 
= \{ 
\text{infinitesimal
graded deformations of $A$}
\}.
$$ 
If an infinitesimal $\mu$
in $\HH^2_{-1}(A)$
is the first multiplication map
$\mu_1$ for a graded deformation
of $A$, then we say $\mu$
{\em lifts (or integrates) to a graded deformation} of $A$.

\subsection*{Koszul sign convention}
We use the {\em Koszul sign convention}
throughout for the tensor product of maps: 
If $f: A \to B$ and $f' : A' \to B'$ are homogeneous maps of graded vector spaces, then $f\ot f':A\ot A'\rightarrow B \ot B'$ is the map
satisfying, for homogeneous $a\in A$ and $a' \in A'$,
\begin{equation}
\label{KoszulSignConvention}
(f \ot f')(a \ot a') = (-1)^{\deg(a)\deg(f')} f(a) \ot f'(a')
\quad
\text{(Koszul sign convention)}
\, .
\end{equation}
%%%%%%%%%%%%%%%
%%%%%%%%%%%%%%%%

%%%%%%%%%%%%%%%
%%%%%%%%%%%%%%%%

\section{Cohomology in the nonmodular setting}
\label{NonmodularSection}
%%%%%%%%%%%%%%%%%%%%%%%
%%%%%%%%%%%%%%%%%%%%%%%
We review a description of the Hochschild cohomology for $A=S(V) \rtimes G$ in the nonmodular setting before investigating
the modular case.
Consider an arbitrary finite group $G\subset \GL(V)$ with
$\cchar \FF$ and $|G|$ coprime
and $\FF$
algebraically closed for $V\cong \FF^n$.
Let $C$ be a set of representatives of the conjugacy classes of $G$ and
let $Z(h)$ be the centralizer of $h$ in $G$.
As $|G|$ is invertible
in $\FF$,
there is a $G$-invariant
inner product on $V$
(obtained by averaging any inner product
over the group $G$).
Then for any $h$
in $G$,
$Z(h)$ preserves set-wise
both $V^h$ and $(V^h)^\perp$,
the orthogonal
complement to $V^h$,
and we may define
the
{\em Hochschild character}
$$\chi_h: 
Z(h) \to \FF,
\quad
z \mapsto \det(z|_{(V^{h})^\perp}),
$$
recording the determinant of
$Z(h)$
acting on 
$(V^h)^\perp$.
Note that
if $G$ is abelian,
we may identify 
the Hochschild character
$\chi_h$
with the linear character of \cref{LinearCharacter}
after identifying
the $\FF G$-modules $(V^h)^\perp$
and $V/V^h$ using the fact that
$V$ decomposes as
$V^h\oplus (V^h)^\perp$ as an 
$\FF G$-module
for $h$ in $G$.

\begin{prop}
\label{NonModularCohDeg2}
Consider a finite group $G\subset\GL(V)$ acting on $V \cong \FF^n$ with $|G|$ coprime to 
$\cchar \FF$ and $\FF$ algebraically closed.
For
$A=S(V)\rtimes G$,
\begin{equation*}
\begin{aligned}
\HH^{\DOT}(A)
\ &\cong\ 
\bigoplus_{h \in C}
\left(
S(V^h) 
\ot
\Wedge^{\DOT\, -\, \codim V^h}
(V^h)^*
\right)^{\chi_h}
\qquad\qquad
\text{and, in particular,}
\\
\HH^{2}(A)
\ &\cong\
\big( S(V)\ot \Wedge^2 V^*\big)^G
\ \oplus \hspace{-2ex}
\ \bigoplus_{
\substack{h \in C \\ \codim V^h=2
\\ \det h=1 \rule{0ex}{1.5ex}
}}
\hspace{-2ex}
\left(
S(V^h) \right)^{\chi_h}
\, .
\end{aligned}
\end{equation*}
\end{prop}
\begin{proof}
The arguments
in \cite[Section 6]{Witherspoon06}
and
\cite[Theorem 3.1]{SW08}
(see also \cite[Example 3.5.7]{W19})
using
\cite{Stefan}
give the
first isomorphism 
(also see \cite{Farinati}).
Note that each space
of $\chi_h$-invariants
is
$\{0\}$ unless
$\det(h)=\chi_h(h)=1$
since
$h$ itself acts
trivially
on $S(V^h) 
\ot
\Wedge^{\DOT\, -\, \codim V^h}
(V^h)^*$
(see \cite[Equation (3.7)]{SW08}).
Since any reflection in $G$
is diagonalizable
with determinant
$\neq 1$,
we thus
need only sum over 
group elements
whose fixed point spaces
have codimension $0$ or $2$
to find
$\HH^2(A)$.
\end{proof}

Under the induced
grading on $\HH^{\DOT}(A)$,
the graded component 
$\HH^m_{i}(A)$
is the subspace of $\HH^m(A)$
consisting of
elements whose first
tensor component has polynomial
degree $i+m$:
%%%
%%%
\begin{prop}
\label{NonModularCohCyclic}
Let $G\subset\GL(V)$ be a finite group acting on $V \cong \FF^n$ with $|G|$ coprime to 
$\cchar \FF$ and $\FF$ algebraically closed.
The space of 
infinitesimal
graded deformations of $A=S(V)\rtimes G$
is isomorphic as an $\FF$-vector space to
%%%
%%
$$
\begin{aligned}
\HH^2_{-1}(A)
\ \cong\ 
(V \ot \Wedge^2V^*)^G
\ \oplus \hspace{-2ex}
\bigoplus_{\substack{h \in G\\ \codim V^h=2\\ \det h =1 \rule{0ex}{1.5ex}}}
\hspace{-2ex}
\left(
V^h
\right)^{\chi_h}
\,.
\end{aligned}
$$
\end{prop}

For $G$ cyclic,
our main finding
\cref{MainThm} recovers
\cref{NonModularCohCyclic}
in this nonmodular setting,
see \cref{RecoverNonModResult}.
Note that the case where $S(V)$
is replaced by a quantum polynomial ring over fields of characteristic $0$ was explored
by Naidu, Shakalli, Shroff, and Witherspoon see \cite{NSW11,Shakalli12,NaW16,ShroffWi16}.
For deformations in that setting,
see also \cite{LevandovskyyShepler,SheplerUhl}.

%%%%%%%%%%%%%%%%%%%%%%%%%%%%%%%%%%%%%%%%%%%%%%%%%%%%%%%%%%%%%%%%%%%%%%%%%%%%%%%%%%%%%%%%%%%%%%%%%%%%
%%%%%%%%%%%%%%%%%%%%%%%%%%%%%%%%%%%%%%%%%%%%%%%%%%%%%%%%%%%%%%%%%%%%%%%%%%%%%%%%%%%%%%%%%%%%%%%%%%%%
\section{Periodic-twisted-Koszul Resolution for cyclic groups}
\label{PeriodicResolutionSection}
%%%%%%%%%%%%%%%%%%%%%%%%%%%%%%%%%%%%%%%%%%%%%%%%%%%%%%%%%%%%%%%%%%%%%%%%%%%%%%%%%%%%%%%%%%%%%%%%%%%%
%%%%%%%%%%%%%%%%%%%%%%%%%%%%%%%%%%%%%%%%%%%%%%%%%%%%%%%%%%%%%%%%%%%%%%%%%%%%%%%%%%%%%%%%%%%%%%%%%%%%
We consider a finite cyclic group $G$
acting linearly on $V \cong \FF^n$
for $\FF$ a field of arbitrary
characteristic.
To compute the Hochschild cohomology of $A=S(V) \rtimes G$, we use a twisted product resolution for $A$:
We twist together a periodic resolution for the group $G$ and the Koszul resolution for $S(V)$. 
See
\cite{SW14,SW19} for the construction details (and requirements) of twisted product resolutions for general skew group algebras. 
We fix a generator $g$ of $G$.

%%%%%%%%%%%%%%%%%%%%%%%%%%%%%%%%%%%%%%%%%%%%%%%%%%%
\subsection*{A periodic resolution for cyclic groups}
%%%%%%%%%%%%%%%%%%%%%%%%%%%%%%%%%%%%%%%%%%%%%%%%%%%
We may identify $\FF G$ with
$\FF[x]/(x^{|G|}-1)$ for some indeterminate $x$
and use a well-known resolution $P_{\DOT}$ of $\FF G$
given by
(see \cite{GGRSV91}) 
\begin{center}    
\noindent 
\begin{tikzcd}
P_{{\bullet}}:\ \  \cdots \arrow[r, "\gamma"] & \FF G \otimes \FF G \arrow[r, "\eta"] &  \FF G \otimes \FF G \arrow[r, "\gamma"] & \FF G \otimes \FF G  \arrow[r, "m"] & \FF G 
\arrow[r] & 0,
\end{tikzcd}
\end{center}
where 
$\gamma = g \otimes 1 - 1 \otimes g$, \ 
$\eta = g^{-1} \ot 1 + g^{-2} \ot g + \cdots + 1 \ot g^{-1}$ 
and $m$ is multiplication.

To satisfy the compatibility requirements for constructing a twisted product resolution
(see \cite[Defintion 2.17]{SW19}),
we use the following $G$-grading on $\FF G \ot \FF G$:
For any $h$ in $G$ and $P_i = \FF G \ot \FF G$
for $i \ge 0$, set
$$
(P_i)_{h}=
\begin{cases}
\Span_{\FF}\{a \ot b \ :\ ab = h \}
& \text{ for $i$ even, }
\\
\Span_{\FF}\{a \ot b \ :\ ab = hg^{-1} \}
& \text{ for $i$ odd}
\,.
\end{cases}
$$

%%%%%%%%%%%%%%%%%%%%%%%%%%%%%
%%%%%%%%%%%%%%%%%%%%%%%%%%%%%
\subsection*{The Koszul complex}
%%%%%%%%%%%%%%%%%%%%%%%%%%%%%
%%%%%%%%%%%%%%%%%%%%%%%%%%%%%
Recall the bimodule Koszul complex for the symmetric algebra $S(V)$: 
\begin{center}
\noindent \begin{tikzcd}
K_{{\bullet}}:
\ \cdots \longrightarrow S(V) 
\otimes \Wedge^2 V \otimes S(V) \longrightarrow  
S(V) \otimes V \otimes S(V) \longrightarrow S(V) \otimes S(V)  \longrightarrow  S(V) \longrightarrow  0
\,,
\end{tikzcd}
\end{center}
with differentials defined, for all $w_1 \wedge \cdots \wedge w_j$ in $\Wedge^j V$, by 
\begin{small}
$$
\begin{aligned}
\del_{K}
(
{\small 1 \ot 
w_1 \wedge \cdots \wedge w_j \ot 1})
&
=
\displaystyle{
\sum_{\ell=1}^j}
(-1)^{\ell-1}
(w_{\ell} \ot 
w_1 \wedge \cdots \wedge \hat{w}_{\ell} 
\wedge \cdots \wedge w_j \ot 1
-
1 \ot 
w_1 \wedge \cdots \wedge \hat{w}_{\ell} 
\wedge \cdots \wedge w_j \ot w_{\ell})
\,.
\end{aligned}
$$
\end{small}%to prevent indent
This resolution satisfies the compatibility requirements 
of \cite{SW19} 
for a twisted product resolution.

%%%%%%%%%%%%%%%%%%%%%%%
%%%%%%%%%%%%%%%%%%%%%%%
%%%%%%%%%%%%%%%%%%%%%%%5
\subsection*{Periodic-twisted-Koszul resolution}
The
periodic-twisted-Koszul resolution \
$X_{\DOT} = P_{\DOT}\ot^G K_{\DOT}$ of $A=S(V)\rtimes G$
is the total complex
(see \cite{SW14,SW19}) 
\begin{equation*}
X_m = \bigoplus_{i+j=m} X_{i,j}
\quad
\text{for}
\quad
X_{i,j} = P_i \ot K_j
=(\FF G \ot \FF G)_i \ot 
\big(
S(V) \ot \Wedge^jV \ot S(V)
\big)
\quad
\text{for $i,j \ge 0$}
\end{equation*}
with $A$-bimodule structure
on each $P_i\ot K_j$
given by
$$
s(y_1\ot y_2)a=y_1a\ot\, ^{(ha)^{-1}}s\ ^{a^{-1}}y_2
\qquad\text{ for }
y_1\in (P_i)_h,\ y_2\in K_j,
\ a,h\in G,\ s\in S(V)
\, 
$$
and differentials
$d_m: X_m \to X_{m-1}$
given by $d_m = \sum_{i+j=m} d^{\text{hor}}_{i,j} + d^{\text{vert}}_{i,j}$ for
horizontal and vertical
maps
%%%%%%%%%
%%%%%
%%
$$
d_{i,j}^{\text{hor}}:=\del_{P} \ot 1_{K}: 
\,
X_{i,j} \to X_{i-1,j}
\ \ \
\text{and}
\ \ \ 
d_{i,j}^{\text{vert}}:=  1_{P} \ot \del_{K}:
\,
X_{i,j} \to X_{i, j-1}
\,.
$$
%%
%%%%
%%%%%%
Here, $\del_{P}$ and $\del_{K}$ 
are the differentials for $P_{\DOT}$ and $K_{\DOT}$, 
respectively.
The complex $X_{\DOT}$ 
gives a free $A$-bimodule resolution of
$A=S(V)\rtimes G$:
\begin{equation*}
\begin{tikzcd}
X_{\DOT}: 
\ \ \ \cdots \arrow[r] & 
X_2 \arrow[r] &  X_1 \arrow[r] & X_0 \arrow[r] &
A \arrow[r] & 0
\, .
\end{tikzcd}
\end{equation*}

We identify 
each $X_{i,j}$ with
$A\ot \Wedge^j V \ot A$ 
using the $A$-bimodule isomorphism
$$
X_{i,j}\
\overset
{\cong}{\longrightarrow}\ \ 
A\ot \Wedge^j V \ot A
$$
given by,
for
$a'\ot a \in (P_i)_h$,
$w_1\wedge\cdots\wedge w_j\in \Wedge^j V$, and $s,r$ in $S(V)$,
$$
(a'\ot a)\ot (s\ot 
w_1\wedge\cdots\wedge w_j\ot r)
\longmapsto
(\, ^{h}s\ot a') \ot 
\, ^a\, 
(w_1\wedge\cdots\wedge w_j)\,
\ot (\, ^a r\ot a)
\, .
$$
Note that after the identifications,
the differential
encodes the group action
although the group algebra $\FF G$
does not appear in $X_{i,j}=A\ot \Wedge^j V \ot A$ as an overt tensor component.

The Hochschild cohomology $\HH^{\DOT}(A)$
is thus the homology
of the complex 
\begin{equation*}
\begin{tikzcd} 
\ \ \ \cdots \arrow[r] & 
\text{Hom}_{A^e}(X_2, A) 
\arrow[r] &  
\text{Hom}_{A^e}(X_1, A) 
\arrow[r] & 
\text{Hom}_{A^e}(X_0, A)\arrow[r] &
 0
\, .
\end{tikzcd}
\end{equation*}
The grading on each $\HH^{m}(A)$
induced by the grading
on $A$ coincides
with the grading
on $\Wedge V$
with $\Wedge^j V$ in degree $j$
so that
$\gamma\in 
\Hom_{A^e}(A\ot \Wedge^j V\ot A, A)$ has degree $i$
when
$\deg (\gamma(a\ot v_1\wedge\cdots\wedge v_j \ot a'))
=i+j$ for $v_1, \ldots, v_j\in V$ and
$a, a'\in A$.
We denote by $\CCC^m_i(A)$,
$\ZZZ^m_i(A)$, and $\BBB^m_i(A)$ 
the spaces of 
$m$-cochains, $m$-cocycles, and $m$-coboundaries, respectively, 
of graded degree $i$
so that
\begin{equation}
\label{CohomologyOfComplex}
\HH^2_{-1}(A)
\ \ 
=\ \ 
\ZZZ^2_{-1}(A)
/
\BBB^2_{-1}(A) \, .
\end{equation}
%
%
%%%%
%%%%%%%%%%%%%%%%%%%%%
%%%%%%%%%%%%%%%%%%%%%

\subsection*{Explicit differential}
We give the differential in 
the resolution $X_{\DOT}$
explicitly.
We use the Koszul sign convention
(\cref{KoszulSignConvention}) 
with respect to homological degree
noting that $\del_{P}$ and $\del_{K}$ each have degree $-1$ while the identity has degree $0$:
For $y_1 \in P_i=\FF G \ot \FF G$
and $y_2$ in $S(V) \ot \Wedge^j V \ot S(V)$,
%%%%
\begin{equation*}
(1\ot \del_{K})(y_1\ot y_2)
= (-1)^{(\deg y_1)(\deg \del_{K})}\big(y_1\ot \del_{K}(y_2)\big)
= (-1)^{i}\big(y_1\ot \del_{K}(y_2)\big)
\,.
\end{equation*}
The horizontal differentials 
$\del_{P}\ot 1_{K}$
on $X_{i,j}$ 
 are defined by
$$
d_{i,j}^{\text{hor}}(\overline{w})
\ =\ 
\begin{cases}
\ \ g \ot w_1 \wedge \cdots \wedge w_j \ot 1 - 1 \ot \ ^g(w_1 \wedge \cdots \wedge w_j) \ot g\ \ \ 
&
\text{if $i$ is odd,}
\\
\ \ 
\displaystyle{
\sum_{\ell=0}^{|G|-1}
g^{-1-\ell} \ot
\ ^{g^{\ell}} (w_1 \wedge \cdots \wedge w_j)
\ot g^{\ell}}
&
\text{if $i$ is even}
\,
\end{cases}
$$
and the vertical differentials $1_{P} \ot \del_{K}$
are defined by
(correcting a 
misprint in 
\cite[Example 4.6]{SW14})
$$
\begin{aligned}
d_{i,j}^{\text{vert}}( \overline{w})
=
\begin{cases} \,
\displaystyle{\ 
\sum_{\ell=1}^j}
(-1)^{\ell}
(\, ^g w_{\ell} \ot 
w_1 \wedge \cdots  \wedge \hat{w}_{\ell} 
\wedge \cdots \wedge w_j \ot 1
-
1 \ot 
w_1 \wedge \cdots \wedge \hat{w}_{\ell} 
\wedge \cdots \wedge w_j \ot w_{\ell})
&
\text{$i$ odd,}
\\
\, 
\displaystyle{\ 
\sum_{\ell=1}^j}
(-1)^{\ell-1}
(w_{\ell} \ot 
w_1 \wedge \cdots \wedge \hat{w}_{\ell} 
\wedge \cdots \wedge w_j \ot 1
-
1 \ot 
w_1 \wedge \cdots \wedge \hat{w}_{\ell} 
\wedge \cdots \wedge w_j \ot w_{\ell})
&
\text{$i$ even},
\end{cases}
\end{aligned}
$$
for 
$\overline{w} =1 \ot w_1 \wedge w_2 \wedge \cdots \wedge w_j \ot 1$ 
in $A \ot \Wedge^j V \ot A$.

%%%%%%%%%%%%%%%%%%%%%
%%%%%%%%%%%%%%%%%%%%%
%%%%%%%%%%%%%%%%%%%%%
%%%%%%%%%%%%%%%%%%%%%
%%%%%%%%%%%%%%%%%%%%%
%%%%%%%%%%%%%%%%%%%%%

%%%%%%%%%%%%%%%%%%%%%%%%%%%%
%%%%%%%%%%%%%%%%%%%%%%%%%%%%
%%%%%%%%%%%%%%%%%%%%%%%%%%%%
%%%%%%%%%%%%%%%%%%%%%%%%%%%%
%%%%%%%%%%%%%%%%%%%%%%%%%%%%
%%%%%%%%%%%%%%%%%%%%%%%%%%%%
%%%%%%%%%%%%%%%%%%%%%%%%%%%%
%%%%%%%%%%%%%%%%%%%%%%%%%%%%
%%%%%%%%%%%%%%%%%%%%%%%%%%%%
%%%%%%%%%%%%%%%%%%%%%%%%%%%%
%%%%%%%%%%%%%%%%%%%%%%%%%%%%
%%%%%%%%%%%%%%%%%%%%%%%%%%%%

%%%%%%%%%%%%%%%%%%%%%%%%%
%%%%%%%%%%%%%%%%%%%%%%%%%
%%%%%%%%%%%%%%%%%%%%%%%%%
%%%%%%%%%%%%%%%%%%%%%%%%%
%%%%%%%%%%%%%%%%%%%%%%%%%
%%%%%%%%%%%%%%%%%%%%%%%%%
%%%%%%%%%%%%%%%%%%%%%%%%%
\section{Cocycle conditions and cohomology decomposition}
\label{DecompositionSection}

We begin our analysis of
the space
of
infinitesimal graded deformations
with initial cocycle conditions.
These conditions allow us to 
decompose the Hochschild cohomology 
into contributions from each
group element.
Again, we consider a finite cyclic group
$G\subset \GL(V)$ acting
on $V \cong \FF^n$ and fix a generator $g$ of $G$
with which to construct the periodic-twisted-Koszul resolution 
$X_{\DOT}$ 
(see \cref{PeriodicResolutionSection})
of $A=S(V)\rtimes G$
giving the space $\HH_{-1}^2(A)$
of infinitesimal
graded deformations of $A$.
We examine the differential
of the 
 resolution $X = P_G \ot^G K_S$ 
of $A$
to determine preliminary cocycle conditions.

%%%%%%%%%%%%%%%%%%%%%%%%%%%%
%%%%%%%%%%%%%%%%%%%%%%%%%%%%
%%%%%%%%%%%%%%%%%%%%%%%%%%%%
%%%%%%%%%%%%%%%%%%%%%%%%%%%%
\subsection*{Decomposing cochains}
The space 
$\CCC^2_{-1}(A)
=
\big(
\Hom_{A^e}(X_2, A)
\big)_{-1}
$
of $2$-cochains of degree $-1$
decomposes
into a space of maps on $V$
and a space of maps on $\Wedge^2 V$:
\begin{equation}
\label{CochainDecomp}
\begin{aligned}
\CCC^2_{-1}(A)
=
\,
\Hom_{\FF}(V,\FF G) 
\oplus
\Hom_{\FF}(\Wedge^2 V, \,  V \ot \FF G)
=
(V^* \ot \FF G)
\oplus
(V \ot \Wedge^2 V^* \ot \FF G )  
\, ,
\end{aligned}
\end{equation}
since $X_2 =(A \ot A) \oplus (A \ot V \ot A) \oplus (A \ot \Wedge^2 V \ot A)$
and $\FF G$ is the degree $0$ component of $A$,
using the usual identifications
(see \cref{identifications}).
Note here that the only $A^e$-homomorphism
$A\ot A\rightarrow A$
of degree $-1$ is the zero map.

%%%%%%%%%%%%%%%%%%%%%%%%%%%%%%%%%%%%%5
%%%%%%%%%%%%%%%%%%%%%%%%%%%%%%%%%%%%%%%
\subsection*{Cochains decomposed 
by group
contribution}

We decompose the vector space of cochains
according to group elements
with an extra {\em shift}
by
the generator $g$ of $G$.
From ~\cref{CochainDecomp},
$$
\begin{aligned}
\CCC^2_{-1}(A) =
\bigoplus_{h \in G}
\left(V^* \ot \FF h \right)
\oplus
\left(V \ot \Wedge^2 V^* \ot \FF h \right)
=
\bigoplus_{h \in G}
\left(V^* \ot \FF hg \right)
\oplus
\bigoplus_{h \in G}
\left(V \ot \Wedge^2 V^* \ot \FF h \right)
\,.
\end{aligned}
$$
For each $h \in G$, we set
$$
\CCC^2_{-1}(h) := \underbrace{
\big(V^* \ot \FF hg \big)}_{\text{$\lambda$-part}} 
\oplus \ \underbrace{\big(V \ot \Wedge^2 V^* \ot \FF h \big)}_{\text{$\alpha$-part}}
\qquad\text{ so that }
\qquad
\CCC^2_{-1}(A) = \bigoplus_{h \in G} \CCC^2_{-1}(h)
\,.
$$
Furthermore, we define the coboundaries and cocycles for $h \in G$ by
$$
\begin{aligned}
\BBB^2_{-1}(h) := \CCC^2_{-1}(h) \cap \BBB^2_{-1}(A)
\qquad
\text{and}
\qquad
\ZZZ^2_{-1}(h) := \CCC^2_{-1}(h) \cap \ZZZ^2_{-1}(A)
\,
\end{aligned}
$$
and define the cohomology for $h$ by
\begin{equation}
\label{IndividualCoh}
\HH^2_{-1}(h)  \ \ 
:= \ \ 
\ZZZ^2_{-1}(h)/\BBB^2_{-1}(h)
\, .
\end{equation}
We justify this terminology (and notation) in \cref{CohDecomp}
below by exhibiting
$\HH^2_{-1}(A)$
as the direct sum
of the $\HH^2_{-1}(h)$.
First, we establish two lemmas
that will be useful here and later.
We consider cocycle conditions
in the first lemma using the transfer map $T$ (see \cref{TransferMap})
and 
consider
coboundary conditions in the second.
We write
any cochain 
$\gamma$  
in  $\CCC^2_{-1}(h)$
as 
$\gamma=(\lambda\ot hg) \oplus (\alpha\ot h)$  
for $\lambda$ in $\Hom_{\FF}(V,\FF)$ and $\alpha$ in $\Hom_{\FF}(\Wedge^2 V,V)$.

\begin{lemma}[Cocycles]
\label{JacobiGrp}
For any $h$ in $G$,
a cochain
$(\lambda\ot hg)
\oplus (\alpha\ot h)$ in  $\CCC^2_{-1}(h)$
 is a cocycle in  $\ZZZ^2_{-1}(h)$ if and only if
$$
\begin{aligned}
&
(1)
&&
0
= \lambda(\im(T))
\quad
\text{\em in $\FF$}
\,,
\\
&
(2)
&&
0
= (\alpha-
\,^{g^{-1}}\alpha)(u \wedge v )
-
\lambda( v )(u - \,^{h}u)
+
\lambda( u )(v - \,^{h}v)
\quad
\text{\em in $V$}
\text{ for all $u,v \in V$, and}
\\
&
(3)
&&
0=
\alpha( u \wedge v )(w-\,^hw)
+
\alpha( v \wedge w )(u-\,^hu)
+
\alpha( w \wedge u )(v-\,^hv)
\quad
\text{\em in $S(V)$}
\text{ for all $u,v,w \in V$.}
\end{aligned}
$$
\end{lemma}
\begin{proof}
We first consider 
a cochain 
$\gamma=\lambda \oplus \alpha$  
in  $\CCC^2_{-1}(A)$
using \cref{CochainDecomp}
for $\lambda$ in $\Hom_{\FF}(V,\FF G)$ and $\alpha$ in $\Hom_{\FF}(\Wedge^2 V,V \ot \FF G)$.
We examine $d\gamma$ for
the differential $d$ on
the resolution $X_{\DOT}$
(see 
\cref{PeriodicResolutionSection})
and conclude that
$\gamma$ lies in
$\ZZZ^2_{-1}(A)$ if and only if
$$
\begin{aligned}
&
\bullet
&&
0=\lambda(\im T)
\,,
\quad
\\
&
\bullet
&&
0
=
g \alpha( u \wedge v) 
-
\alpha(\,^gu \wedge \,^g v ) g
-\,^gu \,
\lambda( v )
+
\lambda(v )u
+
\,^gv\, 
\lambda( u )
-
\lambda( u ) v
\ \text{ for all } u,v \in V,
\text{ and }
\\
&
\bullet
&&
0=
[u, \alpha( v \wedge w )]
+
[v, \alpha( w \wedge u )]
+
[w, \alpha( u \wedge v )]
\ \text{ for all } u,v,w \in V
\, ,
\end{aligned}
$$
where the bracket is the commutator in $A$.
%%%
%%%
Now decompose $\alpha$
and $\lambda$
according to group elements,
\begin{equation*}
\lambda(v) = \sum_{h\in G} \lambda_h(v) h
\qquad
\text{and}
\qquad
\alpha(v \wedge w) 
=
\sum_{h \in G} \alpha_h(v\wedge w)
\ot h
\qquad
\text{for all $v,w$ in $V$}
\, 
\end{equation*}
where $\lambda_h: V\rightarrow \FF$
and $\alpha_h:\Wedge^2V \rightarrow V$,
so that 
$\gamma=\sum_{h\in G}\gamma_h$
for $\gamma_h=(\lambda_{hg}\ot hg)
\oplus (\alpha_h \ot h)$.
We compare coefficients to see that $\gamma$ 
satisfies 
the above three conditions
if and only if
each
$\gamma_h$ 
does
if and only if
each $\gamma_h$
satisfies the conditions in  the lemma. 
Hence $\gamma$ 
lies in $\ZZZ^2_{-1}(A)$ 
if and only if each 
$\gamma_h$ 
lies in $\ZZZ^2_{-1}(A)$
and thus
in $\ZZZ^2_{-1}(h)$.
\end{proof}
%%%%%%%%%%%%%%%%%%%%%%%%%%%5
%%%%%%%%%%%%%%%%%%%%%%%%%%%5
%

\begin{lemma}[Coboundaries]
\label{coboundaries}
For any $h$ in $G$,
a cochain
$(\lambda \ot hg) \oplus (\alpha \ot h)$ in $\CCC^2_{-1}(h)$
is a coboundary
in $\BBB^2_{-1}(h)$ 
if and only if there is
some map $f : V \to \FF$ 
with
$$
\begin{aligned}
\lambda(u)=
f(u - \,^gu)
\quad \ \ 
\text{and}
\quad \ \ 
\alpha(u\wedge v)
=
f(v)(u-\,^hu)-f(u)(v-\,^hv) 
\quad\text{ for all $u,v\in V$}.
\end{aligned}
$$
\end{lemma}
%%%
%%%
\begin{proof}
We first write out
conditions for a generic
coboundary.  Say
$f'$ is a $1$-cochain in $\CCC^1_{-1}(A)=V^*\ot \FF G$
and write 
$f'=\sum_{h\in G} f_h\ot  h$
with each
$f_h\ot h\in 
V^*\ot \FF h\subset \CCC^1_{-1}(A)$.
Then
again using the differential
$d$ on the resolution $X_{\DOT}$
(see \cref{PeriodicResolutionSection}), 
we observe that
$$
df'=\sum_{h\in G}
d(f_h\ot h)
\quad\text{ with }\quad
d(f_h\ot h)
= 
(\lambda_{hg}\ot hg)
\oplus
(\alpha_h \ot h)
\,,
\rule[-3ex]{0ex}{3ex}%strut
$$
for
$$
\lambda_{gh}(u)= f_h(u-\,^gu)
\quad\text{ and }\quad
\alpha_h(u\wedge v)
=f_h(v)(u-\, ^hu)
-
f_h(u)(v-\, ^hv)
\text{ for all } u, v\in V
\, .
$$
Then
each 
$(\lambda_{hg}\ot hg)
\oplus
(\alpha_h \ot h)
$
is a coboundary
in 
$\BBB^2_{-1}(h)=
\CCC^2_{-1}(h)\cap
\BBB^2_{-1}(A)$
and satisfies the
condition in the statement
of the lemma.
Thus for $h$ in $G$, if a cochain
$(\lambda \ot hg) \oplus (\alpha \ot h)$ in $\CCC^2_{-1}(h)$ is
$df'$ for some $1$-cochain $f'$ in $\CCC^1_{-1}(A)$,
then $f'=f\ot h$ for some
$f\in V^*$ satisfying
the conclusion of the lemma.
Conversely,
if there is a function $f$
as in the statement of the lemma, then  $f\ot h$ is a cochain in $C^1_{-1}(A)$
with $d(f\ot h)=
(\lambda \ot hg) \oplus (\alpha \ot h)$.
\end{proof}

\subsection*{Cohomology decomposed 
by group
contribution}

Now we may decompose
cohomology according to
group elements:
\begin{prop}
\label{CohDecomp}
Let $G\subset\GL(V)$ be a finite cyclic group acting on  $V \cong \FF^n$. Then
$$
\BBB^2_{-1}(A) 
\ =\ \bigoplus_{h \in G} \BBB^2_{-1}(h)
\quad
\text{and}
\quad
\ZZZ^2_{-1}(A)
\ =\ 
\bigoplus_{h \in G} \ZZZ^2_{-1}(h)
\, .
$$
Thus the space of infinitesimal 
graded deformations of $A=S(V)\rtimes G$
is
$$
\HH^2_{-1}(A)
\ \cong\ 
\bigoplus_{h \in G}
\HH^2_{-1}(h)
\,.
$$
\end{prop}
\begin{proof}

To verify that
$\BBB^2_{-1}(A) \subset \bigoplus_{h \in G} \BBB^2_{-1}(h)$,
consider $df$ in
$\BBB^2_{-1}(A)$
with $f$ a $1$-cochain in $\CCC^1_{-1}(A)=V^*\ot \FF G$. We saw in the proof
of \cref{coboundaries}
that
$df=\sum_{h\in G}
d(f_h\ot h)$
where
$f=\sum_{h\in G} f_h\ot h$
with each summand
$f_h\ot h$ in $
V^*\ot \FF h\subset \CCC^1_{-1}(A)$
and 
$d(f_h\ot \FF h)$
in $\CCC^2_{-1}(h)\cap \BBB^2_{-1}(A) = \BBB^2_{-1}(h)$.
Thus $df$ lies in 
$
\bigoplus_{h\in G}\BBB^2_{-1}(h)$.
The reverse inclusion is clear.

To verify that 
$\ZZZ^2_{-1}(A) = \bigoplus_{h \in G} \ZZZ^2_{-1}(h)$, 
we refer to the proof of \cref{JacobiGrp}: 
For any
$\gamma=\sum_{h\in G}\gamma_h$ 
in $\CCC^2_{-1}(A)$ 
with
each $\gamma_h$ in $\CCC^2_{-1}(h)$,
 $\gamma$ lies in $\ZZZ^2_{-1}(A)$ if and only if each $\gamma_h$ lies in $\ZZZ^2_{-1}(h)$.
%%%%%%
\end{proof}

%%%%%%%%%%%%%%%%%%%%%
%%%%%%%%%%%%%%%%%%%%%
%%%%%%%%%%%%%%%%%%%%%
%%%%%%%%%%%%%%%%%%%%%
%%%%%%%%%%%%%%%%%%%%%
%%%%%%%%%%%%%%%%%%%%%
\section{Cocycle conditions
in terms of codimension}
\label{CodimSection}
%%%%%%%%%%%%%%%%%%%%%
%%%%%%%%%%%%%%%%%%%%%
%%%%%%%%%%%%%%%%%%%%%
%%%%%%%%%%%%%%%%%%%%%
%%%%%%%%%%%%%%%%%%%%%
%%%%%%%%%%%%%%%%%%%%%

In this section, we detangle
the cocycle conditions
for the space 
$\HH_{-1}^2(A)$ of infinitesimal
graded deformations of $A=S(V)\rtimes G$ for
$G\subset\GL(V)$
a finite cyclic group
acting on  $V \cong \FF^n$.
We again fix a generator $g$
of $G$ to define
the resolution
$X_{\DOT}$, see~\cref{PeriodicResolutionSection}.

%%%%%%%%%%%%%%%%%%%%%%%%%%%
%%%%%%%%%%%%%%%%%%%%%%%%%%%%
%%%%%%%%%%%%%%%%%%%%%%%%%%%%
%%%%%%%%%%%%%%%%%%%%%%%%%%%%
\subsection*{Vector space complements and projections}
%%%%%%%%%%%%%%%%%%%%%%%%%%%%
%%%%%%%%%%%%%%%%%%%%%%%%%%%%
%%%%%%%%%%%%%%%%%%%%%%%%%%%%
%%%%%%%%%%%%%%%%%%%%%%%%%%%%
%
%
We describe
$\HH^2_{-1}(A)$
using a choice of cohomology representatives
depending on
projection maps.
Recall
that a $G$-invariant
inner product on $V$ 
may not exist, but we use the notation of an orthogonal complement 
in any case in analogy with
the nonmodular setting.
We choose a vector space
complement $(V^h)^\perp$
to $V^h$ for each $h$ in $G$
so $V=V^h \oplus (V^h)^\perp$. 
Note that 
$V^g = V^G$ and so this gives a
decomposition $V=V^G \oplus (V^G)^\perp$.
We also
choose
a vector space complement
$(V_h)^\perp$
to
$V_h$
with projection map
$\pi_h: V \rightarrow V_h$:
\begin{equation}
    \label{ChoiceOfComplement}
V=\im(1-h)\oplus \im(1-h)^\perp = V_h \oplus (V_h)^\perp
\, .
\end{equation}

%%%%%%%%%%%%%%%%%%%%%
%%%%%%%%%%%%%%%%%%%%%
\subsection*{Cocycle condition (1)}
%%%%%%%%%%%%%%%%%%%%%
%%%%%%%%%%%%%%%%%%%%%

We interpret the first
cocycle condition of \cref{JacobiGrp}.
Recall that $h$ in $G$ is a {\em reflection}
when the fixed-point space $V^h$ is a hyperplane, i.e., $\codim V^h=1$.
Note that if $h$ is a reflection, then
either $h$
is diagonalizable with
order $|h|$ coprime
to $\cchar \FF$  
or 
$h$ is nondiagonalizable 
with order $|h|=\cchar \FF$
(see \cite{SmithBook}).
Also note that the following lemma fails
when $\cchar \FF=2=n=|G|$,
but we have excluded $\cchar \FF=2$
from consideration throughout.
\begin{lemma}
\label{ReflectionLemma}
If $G$ contains a nondiagonalizable reflection, then $\im T = \{0\}$.
\end{lemma}
\begin{proof}
Let $h$ be a nondiagonalizable
reflection with $H
=\langle h\rangle \subset G$
and let 
$p=\cchar \FF=|h|$.
We write $T^G=\sum_{a\in G} a$
and $T^H=\sum_{a\in H} a$
as transformations on $V$
and note that for
a transversal $g_1, g_2, \dots, g_m$ 
of $H$ in $G$
(coset representatives
for $G/H$)
with $m = [G:H]$
(see \cite[Theorem 4.1]{Smith97}
and \cite{ShankWe99}),
$$
T^G(v) = \sum_{a \in G} \,^a v 
=
\sum_{i=1}^{m} 
\sum_{a\in H}
\,^{g_i a} v
=
\sum_{i=1}^{m} 
\,^{g_i}
\Big(
\sum_{a\in H}
\,^{a} v
\Big)
=
\sum_{i=1}^{m} \,^{g_i}\left( T^H(v) \right)
\text{ for } v\in V
\,.
$$
We argue that $ \im T^G$ is zero by showing $\im T^H$ is zero.  
There is a basis $v_1, \dots, v_n$ of $V$ with
$$
\,^hv_i = v_i 
\quad
\text{for $i<n$}
\qquad
\text{and}
\qquad
\,^hv_n = v_1+ v_n 
\,.
$$
Then $T^H(v_i)=|H| \, v_i = 0$ for $i<n$ and
$$
T^H(v_n) = \sum_{j=0}^{p-1} \,^{h^j} v_n
=
\sum_{j=0}^{p-1}
(jv_1 + v_n)
=
\tfrac{p(p-1)}{2}\ v_1
=
0
\,
\qquad \text{as well.}
$$

\vspace{-3ex}${}_{}$

\end{proof}

\begin{example}{\em 
Note that
$\im T\neq \{0\}$ for
$G\subset \GL_4(\FF_3)$ generated
by
$
\begin{smallpmatrix}
1 & 1 & 0 & 0 \\
0 & 1 & 1 & 0 \\
0 & 0 & 1 & 0 \\
0 & 0 & 0 & -1 \\
\end{smallpmatrix}
\, ,
$
but $\im T = \{0\}$
for 
$G\subset \GL_3(\FF_3)$ generated
by $g=
\begin{smallpmatrix}
1 & 1 & 0 \\
0 & 1 & 0 \\
0 & 0 & -1 \\
\end{smallpmatrix}
\, 
$
by ~\cref{ReflectionLemma}
since $g^4$
is a nondiagonalizable reflection.
}%end em for example
\end{example}

%%%%%%%%%%%%%%%%%%%%%
%%%%%%%%%%%%%%%%%%%%%
\subsection*{Cocycle condition (3)}
%%%%%%%%%%%%%%%%%%%%%
%%%%%%%%%%%%%%%%%%%%%

We also interpret the third
cocycle condition of \cref{JacobiGrp}
for cochains.
Recall that we write
any $\gamma$ in
$C^2_{-1}(h)
=
(V^* \ot \FF hg )
\oplus (V \ot \Wedge^2 V^* \ot \FF h )
$
as 
$\gamma=(\lambda\ot hg) \oplus (\alpha\ot h)$  
for $\lambda$ in $\Hom_{\FF}(V,\FF)$ and $\alpha$ in $\Hom_{\FF}(\Wedge^2 V,V)$.

%%%%%%%%%%%%%
\begin{lemma}
\label{CodimConditions}
For any $h$ in $G$,
if a cochain $\gamma=(\lambda \ot hg) \oplus (\alpha \ot h)$ in $\CCC^2_{-1}(h)$ satisfies \cref{JacobiGrp}(3),
then either
\begin{enumerate}
    \item[(a)] $h$ is the identity element of $G$, or
    
    \item[(b)] $\codim V^h = 1$ and $\alpha(u \wedge v)=0$ for all $u,v$ in $V^h$, or
    
    \item[(c)]
    $\codim V^{h}=2$ and $\alpha(u\wedge v)$ lies in $\FF$-span$\{v-\,^hv\}$ for $u \in V^h$ and $v \in V$, or
    
    \item[(d)]
    $\codim V^h>2$ and $\alpha(u\wedge v)$ lies in $\FF$-span$\{u-\,^hu,\ v-\,^hv\}$ for all $u,v \in V$.
\end{enumerate}
\end{lemma}
\begin{proof}
For brevity, write $\hat u = u- \,^hu$ for all $u \in V$
so that 
\cref{JacobiGrp}(3) implies
that
\begin{equation}
\label{Condition3}
0=\alpha(u \wedge  v)\hat w
+
\alpha(v\wedge  w)\hat u
+
\alpha(w\wedge  u)\hat v
\qquad
\text{ in $S(V)$ }
\text{for all $u,v,w \in V$}
\,.
\end{equation}
If $u,v \in V^h$
and $\codim V^h \geq 1$,
then
$\alpha(u\wedge v)=0$ 
as we may choose $w\notin V^h$.
If $u \in V^h$ 
and $v \in (V^h)^\perp$ 
and $\codim V^h\geq 2$,
then
$\alpha(u\wedge v)
\in \FF\hat{v}$ 
as
we may choose
$w$
with $\hat{v}$ and $\hat{w}$
independent.
Lastly,
if $u,v \in (V^h)^\perp$
with $\codim V^h>2$,
then
$\alpha(u\wedge v)
\in \FF\hat{u}+ \FF\hat{v}$ 
as we may choose $w$
with $\hat w\notin\FF\hat u+\FF \hat v$.
%%%%%%%%%%%%%%%%%%%%%
\end{proof}
%%%%%%%%%%%%%

%%
\begin{lemma}[A partial converse to \cref{CodimConditions}]
\label{JacobiConverse}
For any $h$ in $G$,
suppose $\gamma=(\lambda \ot hg) \oplus (\alpha \ot h)$ is a cochain in $\CCC^2_{-1}(h)$ with either
\begin{enumerate}
    \item[(a)] $h$ is the identity element of $G$, or
    
    \item[(b)] $\codim V^h = 1$ and $\alpha(u\wedge v)=0$ for all $u,v$ in $V^h$, or
    
    \item[(c)]
    $\codim V^{h}=2$ and $\alpha(u\wedge v)=0$ for $u \in V^h$ and $v \in V$, or
    
    \item[(d)]
    $\codim V^h>2$ and $\alpha(u\wedge v) =0$ for all $u,v \in V$.
\end{enumerate}
Then $\gamma$ satisfies \cref{JacobiGrp}(3).
\end{lemma}
%%

%%%%
\begin{proof}
We again write $\hat v = v-\,^hv$ for all $v \in V$ and verify that
\begin{equation}
\label{eqn3}
0=\alpha(u\wedge  v)\hat w
+
\alpha(v\wedge  w)\hat u
+
\alpha(w\wedge  u)\hat v
\qquad
\text{for all $u,v,w \in V$}
\,.
\end{equation}
We assume $\codim V^h$ is $1$ or $2$ 
else the statement is trivial
and consider $V=V^h \oplus (V^h)^\perp$. 
Notice that
the right hand side of
\cref{eqn3} vanishes automatically
for $u,v,w\in (V^h)^\perp$
since it defines an alternating linear function and 
$\Wedge^3 (V^h)^\perp=\{0\}$.
If $\codim V^h=1$,
then the right hand side  also vanishes for
$u,v\in V^h$, $w\in V$
by (b)
and 
for
$u\in V^h$, $v,w \in (V^h)^\perp$ 
since it is alternating in $v,w$ for fixed $u$ and $\Wedge^2 (V^h)^\perp=\{0\}$. 
Lastly, if $\codim V^h=2$,
then the right side of
\cref{eqn3} vanishes
for
$u \in V^h$, $v,w \in V$
as (c) implies that
$\alpha(u\wedge v)=0=\alpha(u\wedge w)$.
\end{proof}
%%%%

%%%%%%%%%%%%%%%%%%%%%%%%%%%%
%%%%%%%%%%%%%%%%%%%%%%%%%%%%
%%%%%%%%%%%%%%%%%%%%%%%%%%%%
%%%%%%%%%%%%%%%%%%%%%%%%%%%%
%%%%%%%%%%%%%%%%%%%%%%%%%%%%
%%%%%%%%%%%%%%%%%%%%%%%%%%%%

%%%%%%%%%%%%%%%%%%%%%%%%%%%
%%%%%%%%%%%%%%%%%%%%%%%%%%%
\section{Unique cohomology representatives}
\label{UniqueRepsSection}
%%%%%%%%%%%%%%%%%%%%%%%%%%%%%
%%%%%%%%%%%%%%%%%%%%%%%%%%%%%

In this section, we identify unique
representatives
for the cohomology classes
in the space $\HH^2_{-1}(A)$
of infinitesimal
graded deformations of $A=S(V)\rtimes G$
for a finite cyclic group
$G\subset \GL(V)$
generated by $g$
acting on $V \cong \FF^n$.
By \cref{CohDecomp},
$$
\HH^2_{-1}(A)
\ =\ 
\bigoplus_{h \in G}
\HH^2_{-1}(h)
\,
\qquad\text{ with }
\qquad
\HH^2_{-1}(h)
\ =\ \ZZZ^2_{-1}(h)/\BBB^2_{-1}(h)
\, ,
$$
and we describe  
coset representatives for
each 
$\HH^2_{-1}(h)$
using a choice of vector space
complement
$(V_h)^\perp$
to $V_h$
and complement $(V^h)^\perp$
to $V^h$ 
(see \cref{ChoiceOfComplement})
with projection map
$\pi_h: V\rightarrow V_h$.
Recall that 
the space of cochains for each $h$ in $G$
is
$\CCC_{-1}^2(h)=(V^* \ot \FF gh) 
\oplus (V \ot \Wedge^2 V^* \ot \FF h )$.

\begin{prop}
\label{UniqueReps}
Fix $h \in G$.
Each coset in $\HH^2_{-1}(h)$ 
has a unique representative
$\gamma = (\lambda \ot hg) \oplus (\alpha \ot h)$ in $\ZZZ^2_{-1}(h)$ 
for $\lambda\in V^*$ and
$\alpha\in \Hom_{\FF}(\Wedge^2 V, V)$
with
$\pi_h\alpha\equiv 0$ 
and

\begin{enumerate}
 \setlength{\itemindent}{-1.5ex}
\item[(a)]
when $\codim V^h =0$ (so $h = 1_G$),
$\lambda\equiv 0$ on 
$(V^G)^{\perp}$,
\item[(b)]
    when $\codim V^h=1$, $\alpha(u \wedge v) = 0$ for all $u,v \in V^h$ and 
    $\chi_h$ nontrivial implies $\lambda \equiv 0$ on $(V^h)^\perp$,
\item[(c)]
when $\codim V^h =2$,
$\alpha(u\wedge v)=0$ 
  for all $u \in V^h$, $v \in V$
  and
  $\lambda \equiv 0$ on $V^h$,
  and
    \item[(d)]
when    $\codim V^h >2$, 
$\gamma=0$.
\end{enumerate}
\end{prop}
\begin{proof}
Fix $\gamma=(\lambda \ot hg) \oplus (\alpha \ot h)$
in $\ZZZ^2_{-1}(h)$.
We show $\gamma$ is in the same
coset as a cocycle 
$\gamma-df$
satisfying
the given conditions
and then show this cocycle is unique.
We construct the 
linear function
$f:V\rightarrow \FF$ explicitly
and write
$\gamma-df=\gamma'=
(\lambda'\ot hg)\oplus (\alpha'\ot h)$.
Recall that we identify
$\lambda$ with a function
$\lambda: V\rightarrow \FF$
and $\alpha$ with a function $\alpha:\Wedge^2 V\rightarrow V$.
Set $\hat{v}=v-\, ^hv$
for any $v$ in $V$.

\vspace{1ex}

\noindent
{\bf Existence of representatives.}
Assume $\codim V^h=0$, i.e., $h=1_G$. 
For $v\in (V^G)^\perp$,
define $f(v-\,^gv)=\lambda(v)$
and extend to a linear function
$f: V\rightarrow \FF$.
Note that $f$ is well-defined:
If $v-\,^gv = w-\,^gw$
then $v-w\in V^G$, so $v=w$
for $v,w$ in $(V^G)^\perp$.
Then $df\in \BBB^2_{-1}(1_G)$
and $\gamma'=\gamma-df$ satisfies the conditions
in the statement by \cref{coboundaries} as
$\lambda'(v)=\lambda(v)-f(v-\,^gv)=0$
for $v\in (V^G)^\perp$.

%%%
Now assume
$\codim V^h=m \geq 2$
with $v_1, \ldots, v_m$
a basis of $(V^h)^\perp$
and observe that $\hat{v}_1,\ldots, \hat{v}_m$
form a basis of $V_h$.
Define a map
$f:(V^h)^\perp\rightarrow \FF$
by setting $f(v_i)$ and $f(v_j)$
to be the unique constants such that
$$
\pi_h\alpha(v_i \wedge v_j)=f(v_j)\hat{v}_i
- f(v_i) \hat{v}_j
\, 
\quad\text{ for $1\leq i\neq j \leq m$}
$$
using \cref{JacobiGrp}
and \cref{CodimConditions}(d)
when $\codim V^h>2$
and using
the fact that
$V_h=\FF\text{-span}\{\hat{v}_1,\hat{v}_2\}$
when $\codim V^h=2$.
Notice that $f$ is well-defined:
If $\pi_h\alpha(v_i \wedge v_j)=a\hat{v}_i+b\hat{v}_j$
and 
$\pi_h\alpha(v_i \wedge v_k)=c\hat{v}_i+d\hat{v}_k$
for $a,b,c,d$ in $\FF$
and $i,j,k$ distinct,
then
$
(a-c)\hat{v}_i + b\hat{v}_j-d\hat{v}_k
=
\pi_h\alpha(v_i \wedge (v_j-v_k))
$
is a linear combination
of
$\hat{v}_i$ and $\hat{v}_j-\hat{v}_k$
by \cref{CodimConditions}(d)
so $b=d$.
By \cref{CodimConditions},
we may extend $f$ to a map $f:V\rightarrow \FF$
satisfying
$$
\alpha(v\wedge u)=f(u)\hat v
\quad\text{ for }u\in V^h,\ v\in V
\, .
$$
Again, this is well-defined
as $\alpha((v-w)\wedge u)$
lies in the span of $\hat{v}-\hat{w}$
for all $u\in V^h$, $v,w$ in $V$.

We claim that $\gamma-df=\gamma'=(\lambda'\ot hg) \oplus (\alpha'\ot h)$
satisfies the conditions in the statement.
By \cref{coboundaries},
\begin{equation}
\label{Eqn0}
\alpha'(v\wedge u)
=\alpha(v\wedge u)-
f(u)\hat{v}+f(v)\hat{u}
\quad\text{ for }
u,v\in V
\, .
\end{equation}
Thus $\alpha'(v\wedge u)$
vanishes for
$u,v\in V^h$ as
$\alpha(v\wedge u)=0$
by \cref{JacobiGrp,CodimConditions}
and also vanishes for $u\in V^h$, $v\in (V^h)^\perp$
as
 $\alpha(v\wedge u)=f(u)\hat{v}$
 by construction of $f$.
For $u,v\in (V^h)^\perp$, 
$$
\pi_h\alpha'(v\wedge u)
=
\pi_h\alpha(v\wedge u)-
f(u)\hat{v}+f(v)\hat{u}
=
\pi_h\alpha(v\wedge u)-\pi_h\alpha(v\wedge u)
=0 \, 
$$
as all functions involved are linear:
For
$v=\sum_i a_i v_i$ and $u=\sum_j b_j v_j$
with $a_i, b_j\in \FF$ 
$$
\alpha(v\wedge u)
=
\sum_{i,j} a_i b_j(f(v_j) \hat{v}_i - f(v_i)\hat{v}_j)
=f(u)\hat{v}-f(v)\hat{u
}
\, .
$$
Hence $\pi_h\alpha'\equiv 0$.
Note in particular that this implies
that
$\alpha'\equiv 0$ when $\codim V^h>2$
by \cref{CodimConditions}(d).

Since $\gamma-df=\gamma'$ is a cocycle,
\cref{JacobiGrp}(2) 
implies that 
\begin{equation}
\label{star00}
(\alpha'-\, ^{g^{-1}}\alpha')
(u\wedge  v)
=
\lambda'(v)\hat{u}
-
\lambda'(u)\hat{v}
\quad\text{ for }
u,v \in V
\, .
\end{equation}
Observe in particular that
this implies
$\lambda'\equiv 0$ when $\codim V^h>2$:
In that case, $\alpha'\equiv 0$ so
$$
0 =
\lambda'(v)\hat{u}
-\lambda'(u)\hat{v}
\quad
\text{for all $u,v \in V$}
\, ;
$$
to see that $\lambda'(u)=0$,
choose $v \notin V^{h}$ when
 $u \in V^{h}$
 and choose $v \notin V^{h}$ 
with $\hat u$ and $\hat v$  linearly independent when $u \notin V^{h}$. 
Thus $\gamma-df=\gamma'\equiv 0$ when $\codim V^h>2$.

Now we argue that $\lambda'\equiv 0$ on $V^h$
when $\codim V^h=2$.
Fix $u\in V^h$.
We saw above
that \cref{Eqn0} implies that 
 $\alpha'(u\wedge v)=0$
 for all $v\in V$
and so $\alpha'(\,^gu\wedge \,^gv)=0$ 
for all $v$ as well
since $G$ preserves $V^h$ 
set-wise.
Thus 
by \cref{star00} above,
$\lambda'(u) \hat v
=0$ for all $v$
and $\lambda'(u)=0$.
Thus $\lambda'\equiv 0$ on $V^h$.

%%%%
Lastly, suppose $\codim V^h=1$.
Fix nonzero $x\in (V^h)^\perp$
so $\hat{x}$ spans $V_h\supset \im\pi_h\alpha$ 
and
set $w=x-\, ^gx$.
Let $f: V\rightarrow \FF$
be the linear function with
$$
\begin{aligned}
f(u)\, \hat{x}  
&= \pi_h \alpha(x\wedge u)
\vphantom{\big( a \big)}
&&\quad\text{ for }
u\in V^h,\\
f(x) \, \hat{x}  
&= \big(1-\chi_h(g)\big)^{-1}
\big(\lambda(x)\, \hat{x}  -\pi_h \alpha(x\wedge w)\big)
&&\quad\text{ for }
\chi_h\not\equiv 1,\\
f(x) \ \ 
&= 0 
\rule{0ex}{2.5ex}%strut
&&\quad\text{ for }
\chi_h\equiv 1
\, .
\end{aligned}
$$
We verify that $\gamma-df
=\gamma'=(\lambda'\ot hg)\oplus (\alpha'\ot h)$
satisfies the conditions in the statement.
By \cref{coboundaries},
$$
\alpha'(u\wedge v)=\alpha(u\wedge v)-f(v)\hat{u}+f(u)\hat{v}
\quad\text{ for }
u,v\in V
\, .
$$
Thus $\alpha'(u\wedge v)$ is zero for $u,v\in V^h$
as $\alpha(u\wedge v)=0$ 
by \cref{JacobiGrp,CodimConditions}.
It also vanishes for $u,v\in (V^h)^\perp$
as it defines 
an
alternating function in $u$ and $v$
and $\Wedge^2(V^h)^\perp=\{0\}$
as $\codim V^h=1$.
And for $u\in V^h$, 
$$\alpha'(x\wedge u)
=\alpha(x\wedge u)-f(u)\hat{x}
=\alpha(x\wedge u)
- \pi_h\alpha(x\wedge u)
\, .
$$
Hence $\pi_h\alpha'\equiv 0$.

Now assume $\chi_h$ is nontrivial
so $\chi_h(g)\neq 1$.
We argue that  $\lambda'\equiv 0$
on $(V^h)^\perp$.
First note that $\, ^gx=u+\chi_h(g)x$
for some $u\in V^h$
so that
$w=x-\, ^gx
=u+(1-\chi_h(g))\, x$.
Then as
$\alpha(x\wedge u)
=\alpha(x\wedge w)$,
$$
f(w)\, \hat{x}
=f(u)\, \hat{x}
+\big(1-\chi_h(g)\big)\, f(x)
\, \hat{x}
=\pi_h\alpha(x\wedge u)
+ \lambda(x) \, \hat{x}
-\pi_h \alpha(x\wedge w)
=
\lambda(x)\, \hat x
\, 
$$
and
\cref{coboundaries} implies that
$$
\begin{aligned}
\lambda'(x)\, \hat{x}
=
\big(\lambda(x)-f(x-\, ^gx)\big)\, \hat{x}
=
\lambda(x)\, \hat{x}
-f(w)\, \hat{x}
=0
\, .
\end{aligned}
$$

\noindent
{\bf Uniqueness
of representatives.}
We argue these coset representatives are unique.
Suppose $\gamma=(\lambda \ot hg) \oplus (\alpha \ot h)$ and $\gamma'=(\lambda' \ot hg) \oplus (\alpha' \ot h)$ lie in the same coset of $\HH^2_{-1}(h)$ and
both
satisfy the conditions in the statement. 
Then $\gamma-\gamma' = df$ for some $f: V \to \FF$ and
\cref{coboundaries}
implies that
\begin{equation*}
(\alpha-\alpha')
(u \wedge  v) 
=
f(v)\hat{u}
-
f(u)\hat{v}
\ \ 
\text{ and }
\ \ 
(\lambda-\lambda')
(u)
=f(u-\,^gu)
\quad\text{ for
all $u,v\in V$.}
\end{equation*}
Then
$\im (\alpha-\alpha')\subset V_h$ but $\pi_h (\alpha - \alpha')\equiv 0$ by
assumption,
so $\alpha\equiv\alpha'$
and
\begin{equation}
\label{star1}
0=f(v)\hat{u}
-
f(u)\hat{v}
\quad\text{ for
all $u,v\in V$.}
\end{equation}
To show that $\lambda\equiv\lambda'$,
we argue that $f(u-\, ^g u)=0$
for all $u$ by considering the codimension of $V^h$.

Assume $\codim V^h \ge 2$
and
consider some nonzero
$w=u-\, ^g u$ in $V$.
By \cref{star1},
$$
0=f(v)\hat{w}
- f(w)\hat{v}
\quad\text{ for all }
v\in V
\, .
$$
To see that $f(w)=0$,
choose any $v\notin V^h$
when 
 $w \in V^h$
 and
 choose $v$
with $\hat{v}$ and $\hat{w}$ independent in $V_h$
when $w\notin V^h$.
Thus $(\lambda-\lambda')(u)
=f(w)=0$ and
$\lambda\equiv\lambda'$.

Now assume that $\codim V^h=1$ and 
let $x$ span $(V^h)^\perp$.
First notice that $f$ is zero on $V^h$ 
since \cref{star1} implies that
$
0=
f(u)\hat{x}
$
 for all $u\in V^h$.
Then
$f(u-\, ^g u)=0$ for all $u$ in $V^h$ as $G$ fixes $V^h$ set-wise.
We show $f(x-\, ^g x)=0$ as well.
As above,
$x-\, ^g x=(1-\chi_h(g))\, x$
modulo $V^h$
and thus
$f(x-\, ^g x) = 
(1-\chi_h(g)) f(x)$.
This is zero when 
$\chi_h$ is the trivial character of course, and 
when $\chi_h$ is nontrivial,
then
$\lambda$ and $\lambda'$ both vanish
on $(V^h)^\perp$ by condition (b)
 so already
$0=(\lambda-\lambda')(x)=f(x-\, ^g x)$.
Hence 
$\lambda \equiv \lambda'$.

Finally, assume $h=1_G$. 
By assumption, $\lambda$ and $\lambda'$ are zero on $(V^G)^\perp$ and $\lambda \equiv \lambda'$ on $V^G$ since $(\lambda-\lambda')(v)=f(v-\, ^gv)$
for all $v$.
Hence $\lambda \equiv \lambda'$.
\end{proof}

%%%%%%%%%%%%%%%%%%%%%%%%%%%5
%%%%%%%%%%%%%%%%%%%%%%%%%%%
%%%%%%%%%%%%%%%%%%%%%%%%%%%
\section{The main result: Graded deformation cohomology}
\label{MainTheoremSection}
%%%%%%%%%%%%%%%%%%%%%%%%%%%
%%%%%%%%%%%%%%%%%%%%%%%%%%%
We are now ready 
to describe 
 the space of infinitesimal
graded deformations
for a finite cyclic group $G\subset \GL(V)$ with $V\cong \FF^n$ acting on a polynomial ring $S(V)$
over an arbitrary field $\FF$
with $\cchar \FF \neq 2$.
We describe the Hochschild cohomology
$\HH_{-1}^2(A)$
for $A=S(V)\rtimes G$
in terms of 
the subspaces $V^h$ and $V_h=\im(1-h)$ of $V$ 
for $h$ in $G$
which are stabilized set-wise by $G$
(see \cref{LinearCharacter})
and 
the linear character
$\chi_h:G\rightarrow \FF^{\times}$ 
defined by
$\chi_h(g):=
\det
\left[
\begin{matrix}
g
\end{matrix}
\right]_{V/V^h},
$
for each $h$ in $G$.
In addition, we take the trivial action of
$G$ on $\FF$ so
$\FF^{\chi_h}=\{0\}$
unless $\chi_h\equiv 1$,
the trivial character.
We describe the 
cohomology giving all
graded deformations
of first order
in terms of these linear characters, compare
with \cref{NonModularCohDeg2},
and establish the theorem
in the introduction.
Again we use the transfer map on $V$
given by
$T:V\rightarrow V$, $v\mapsto \sum_{h\in G}\, ^hv$,
which is zero when $G$ contains
a nondiagonalizable reflection
by \cref{ReflectionLemma}.
%%%
%%%
\begin{thm}
\label{MainThm}
Let $G\subset \GL(V)$ be a finite cyclic group 
acting on $V \cong \FF^n$.
The space of 
infinitesimal
graded deformations of $A=S(V)\rtimes G$
is isomorphic as an $\FF$-vector space to
$$
\begin{aligned}
\HH_{-1}^2(A)
\ \cong\
(V^G/\im T)^*
\oplus
\big( 
V \ot\, \Wedge^2 V^*
\big)^G
\   \oplus\hspace{-2ex}
\bigoplus_{\substack{h \in G \\ \codim V^h =1 }}
\hspace{-2ex}
\left(\FF \oplus
\big(V/V_h \ot (V^h)^*\big)
\right)^{\chi_h}
\ \oplus\hspace{-2ex}
\bigoplus_{\substack{h\in G \\ \codim V^h =2 }}
\hspace{-2ex}
\left(
V/V_h
\right)^{\chi_h}
.
\end{aligned}
$$
\end{thm}
%%%
%%%

%%
\begin{proof}
We choose a generator $g$ of $G$
and construct the periodic-twisted-Koszul resolution 
$X_{\DOT}$ 
(see \cref{PeriodicResolutionSection})
of $A=S(V)\rtimes G$
to express $\HH^{\DOT}(A)$.
We then use \cref{CohDecomp}
to
decompose $\HH^2_{-1}(A)$ according to the contribution of each group element:
$$\HH^2_{-1}(A)
\ \cong \
\bigoplus_{h\in G}\HH^2_{-1}(h)
\, .
\, 
\vspace{-1ex}
$$
By \cref{UniqueReps}, $\HH^2_{-1}(h)=0$ when $\codim V^h >2$ and we analyze the remaining cases
to show
$$
\begin{aligned}
\HH^2_{-1}(h) 
\ \ \cong\ \ 
\begin{cases} 
\ \ (V^G/\im T)^*
\oplus
( 
V \ot\, \Wedge^2 V^*)^G
&
\text{when $\codim V^h = 0$,}
\\
\ \ \left(\FF \oplus
\big(V/V_h \ot (V^h)^*\big)
\right)^{\chi_h}
&\text{when $\codim V^h = 1$,
}
\\
\ \ 
\left(
V/V_h 
\right)^{\chi_h}
&\text{when  $\codim V^h = 2$}
\, .
\end{cases}
\end{aligned}
$$
In each case,
we establish the isomorphism
by defining a map 
$\Phi$ from
a set of 
distinguished (cocycle) coset representatives $$
\gamma = (\lambda \ot hg) \oplus (\alpha \ot h)
\quad\text{ in }
\ZZZ^2_{-1}(h)
\qquad\text{ for }
\lambda\in V^*
\text{ and }
\alpha\in \Hom_{\FF}(\Wedge^2 V, V)
\, 
$$
of $\HH_{-1}^2(h)$
given in \cref{UniqueReps}
to the indicated vector space
and then construct a map $\Phi'$ 
the opposite direction with 
$\Phi \Phi'=\Phi' \Phi=
1$, the identity map.
Note that
this choice of representatives
depends on a choice
of
vector space complement $(V^h)^\perp$
to $V^h$ 
and complement $(V_h)^\perp$ to $V_h$
for each $h$ in $G$, see
\cref{ChoiceOfComplement}.

\vspace{1.5ex}

\noindent
{\bf Contribution of the identity:}
For $h=1_G$, define
$$
\begin{aligned}
\Phi:
\HH^2_{-1}(1_G)
\longrightarrow
(V^G/\im T)^* \oplus (V \ot \Wedge^2 V^*)^G
\quad
\text{by}
\quad
(\lambda\ot g) \oplus (\alpha\ot 1_G) + \BBB^2_{-1}(1_G)
\mapsto
\lambda'
\oplus \alpha
\end{aligned}
$$
for 
$\lambda':
V^G/\im T\rightarrow \FF$
the extension of
$\lambda|_{V^G}$ 
to $V^G/\im T$ where $(\lambda\ot g) \oplus (\alpha\ot 1_G)$ is a 
distinguished
coset representative of $\HH^2_{-1}(1_G)$
as in \cref{UniqueReps}.
Note here that
$\lambda|_{\im(T)}\equiv 0$
by \cref{JacobiGrp}(1) so
$\lambda'$ is well-defined.
Also observe that
$\Phi$ has the indicated codomain
because
$\alpha$ lies in $(V \ot \Wedge^2 V^*)^G$
by
\cref{JacobiGrp}(2):
$$
0 = \,^g\left(\alpha(u\wedge v) \right)
-
\alpha(\,^gu\wedge \,^gv)
\quad
\text{for all $u, v \in V$}
\, .
$$
Now we construct an inverse map. Define
$$
\Phi': 
\HH^2_{-1}(1_G) \longleftarrow
(V^G/\im T)^* \oplus (V \ot \Wedge^2 V^*)^G
\quad
\text{by}
\quad
\big((\lambda \ot g) \oplus (\alpha  \ot 1_G)\big) + \BBB^2_{-1}(1_G)
\leftmapsto
\lambda' \oplus \alpha
\,,
$$
 where $\lambda:V\rightarrow \FF$
is defined by
$\lambda(u)=\lambda'(u+\im T)$
for $u$ in $V^G$
and $\lambda\equiv 0$ on $(V^G)^\perp$.
We verify that $(\lambda \ot g) \oplus (\alpha  \ot 1_G)$ lies
in $\ZZZ^2_{-1}(1_G)$
and thus $\Phi'$ is well-defined
by checking
the three cocycle conditions in \cref{JacobiGrp}:
Condition $(1)$ holds 
by construction of $\lambda$
and Conditions (2) and (3)
hold since $h=1_G$ and
$\alpha$ is invariant.
In fact, we observe
that $\Phi'(\lambda' \oplus \alpha)$
is a distinguished
coset representative as in
\cref{UniqueReps}
and one may verify that $\Phi \Phi'=\Phi' \Phi=1$
using \cref{UniqueReps}.

\vspace{1.5ex}

\noindent
{\bf Codimension one contributions:}
Fix $h$ in $G$ with $\codim V^h=1$ and 
consider $V = V^h \oplus (V^h)^\perp$ with fixed element $x$ spanning $(V^h)^\perp$.
Define
$$
\Phi: \HH^2_{-1}(h) \longrightarrow \left(\FF \oplus 
(V/V_h \ot (V^h)^*) \right)^{\chi_h}
\quad
\text{by}
\quad
(\lambda \ot hg)
\oplus (\alpha\ot h)
+  \BBB^2_{-1}(h)
\mapsto
\lambda(x) \oplus
\alpha'
\,,
$$
for $\alpha'$ in
$V/V_h \ot (V^h)^*$
the map $V^h\rightarrow V/V_h$ defined by
$$\alpha'(u)=\alpha(x\wedge u)
+ V_h
\qquad
\text{for $u\in V^h$}
\, 
$$
where $(\lambda \ot hg) \oplus (\alpha \ot h)$ is a
distinguished coset representative of $\HH^2_{-1}(h)$ as  in \cref{UniqueReps}.

We verify that each $\lambda(x)\oplus \alpha'$ is $\chi_h$-invariant and hence $\Phi$ has the indicated codomain.
First notice that $\lambda(x)\in \FF$
is $\chi_h$-invariant since
$G$ acts trivially
on $\FF$
and 
$\lambda(x)=0$
when $\chi_h$ is not 
the trivial character.
Now we argue $\alpha'$ is $\chi_h$-invariant. 
Since $\alpha\equiv 0$ 
on $\Wedge^2 V^h$
and $\, ^gx= \chi_h(g)\, x$ modulo
$V^h$ (as $\dim_\FF (V/V^h)=1$),
\begin{equation}
\label{Eqn1a}
\chi_h(g)\, \alpha(x\wedge u) = \alpha\big(\chi_h(g) x\wedge u\big)  
= \alpha(\,^gx\wedge u)
\quad
\text{  for all $u \in V^h$}
\, ,
\end{equation}
and, as $G$ preserves $V_h$
and $V^h$ set-wise,
\begin{equation}
\label{Eqn1b}
\begin{aligned}
\big(\alpha'-\chi_h(g) \ ^{ g^{-1}}\alpha'\big)(u)
&=
\alpha'(u)
-
\chi_h(g) \ \,^{g^{-1}}\big(\alpha'(\,^gu)\big)
=
(\alpha(x\wedge u)+V_h)
- \chi_h(g) \ ^{g^{-1}}\big(\alpha(x\wedge \,^gu) + V_h\big)
\\
&=
\big(
\alpha(x\wedge u)-\,^{g^{-1}}(\alpha(\,^gx\wedge \,^gu))
\big)+ V_h
=(\alpha-\,^{g^{-1}}\alpha)
(x\wedge u) + V_h 
\quad\text{ for }
u\in V^h
\, .
\end{aligned}
\end{equation}
But
 \cref{JacobiGrp}(2) implies that $(\alpha-\, ^{g^{-1}}\alpha)(u\wedge v)$
lies in $V_h$
for all $u,v$ in $V$
so this last expression
is zero and
thus
$\alpha'$ is also $\chi_h$-invariant.

Now we construct an inverse map to $\Phi$,
$$
\Phi': \HH^2_{-1}(h) \longleftarrow \left(\FF \oplus 
(V/V_h \ot (V^h)^*) \right)^{\chi_h},
\quad
\text{}
\quad
\big(
(\lambda \ot hg) \oplus (\alpha \ot h)\big) + \BBB^2_{-1}(h) \leftmapsto
\lambda' \oplus \alpha'
\, .
$$
For a pair $\lambda'$ in $\FF$
and $\alpha':V^h\rightarrow V/V_h$,
define $\alpha:\Wedge^2 V \rightarrow V$
and $\lambda: V\rightarrow \FF$ as follows.
Let $\alpha(x\wedge u)$ 
be the unique
coset representative in $(V_h)^\perp$
of $\alpha'(u)$ for $u\in V^h$
and
extend to a linear function on $\Wedge^2 V$ 
by setting
$\alpha\equiv 0$ on $\Wedge^2 V^h$.
Then for $u \in V^h$,
$\pi_h \alpha(x\wedge u) = 0$
(for
projection map $\pi_h:V\rightarrow V_h$)
and $\alpha'(u)=\alpha(x\wedge u)
+ V_h$, 
and
as $\alpha$ 
vanishes on $\Wedge^2 V^h$,
\cref{Eqn1a} and \cref{Eqn1b} 
hold and
imply that
\begin{equation}
\label{Eqn4}
\begin{aligned}
(\alpha- \,^{g^{-1}}\alpha)(x\wedge u) +V_h
=
(\alpha'
-\chi_h(g)\ \,^{g^{-1}}
\alpha')(u)
= 0
%\quad\text{
% for all $u$ in $V^h$}
\, ,
\end{aligned}
\end{equation}
i.e.,
$(\alpha - \,^{g^{-1}}\alpha)(x\wedge u)$
lies in 
$V_h=\FF\text{-span}\{x-\,^hx\}$.
Thus we can
define $\lambda: V \to \FF$ 
as the linear function satisfying
$$
\lambda(u)(x-\,^hx)
=(\alpha - \,^{g^{-1}}\alpha)(x\wedge u) \, 
\quad\text{for $u$ in $V^h$}
\qquad
\text{and}
\qquad
\lambda(x) = \lambda'
\,.
$$

We argue that $(\lambda \ot hg) \oplus (\alpha \ot h)$ is a cocycle by checking the conditions in \cref{JacobiGrp}.
For \cref{JacobiGrp}(1), we verify
that $\lambda(\im T)=0$.
Since $\codim V^h =1$,
$h$ is a reflection and
either $|h|=p$ or $|h|$ and $p$
are coprime for $p=\cchar \FF$ 
(see \cref{CodimSection}).
If $|h|=p$, then $h$ is  nondiagonalizable and $\lambda(\im T)=0$
 by \cref{ReflectionLemma}.
Assume now that $p$ and $|h|$
are coprime.
Then $h$ is diagonalizable
with $V^h\cap V_h=\{0\}$, and we may 
assume $(V_h)^{\perp}$ is chosen
as $V^h$
in the construction
of coset representatives 
from \cref{UniqueReps}.
We argue that $(\alpha-\, ^{g^{-1}}\alpha)(x\wedge u) = 0$ for all $u$ in $V^h$.
On one hand,
$(\alpha-\, ^{g^{-1}}\alpha)(x\wedge u)$
lies in $V_h$ by \cref{Eqn4}.
On the other hand, we claim that
$(\alpha-\, ^{g^{-1}}\alpha)(x\wedge u)$
lies in $V^h$.
By construction, $\pi_h \alpha \equiv 0$ and so
$V^h=(V_h)^{\perp}$ contains
both
$\alpha(x\wedge u)$
and 
$\alpha(\, ^gx\wedge \, ^gu)$
and also
$({\,^{g^{-1}}}\alpha)(x\wedge u)
=\,^{g^{-1}}\big(\alpha(\, ^gx\wedge \, ^gu)\big)$ as $G$ preserves
$V^h$; 
 hence the difference
$(\alpha-\, ^{g^{-1}}\alpha)(x\wedge u)$
lies in $V^h$.
Then, as $V^h \cap V_h=\{0\}$, we must have 
$(\alpha-\,^{g^{-1}}\alpha)(x
\wedge u)=0$.
Hence $\lambda(\im T)=0$
as $\im T\subset V^G\subset V^h$
and
\cref{JacobiGrp}(1) holds.

We now verify
\cref{JacobiGrp}(2),
i.e.,
\begin{equation*}
%\label{Eqn5}
(\alpha-\, ^{g^{-1}}\alpha)
(u\wedge  v)
=
\lambda(v)(u-\, ^{h} u)
-
\lambda(u)(v-\, ^{h} v)
\quad\text{ for all }
u,v \in V
\, .
\end{equation*}
The equality holds for $u = x$
and $v \in V^h$ by definition of $\lambda$.
For $u,v \in V^h$, the right-hand-side 
 is zero and
the left side
 is zero as well
since $\alpha$
vanishes on $\Wedge^2 V^h$
by construction
and $G$ fixes $V^h$
set-wise, so
$\alpha(\,^gu\wedge \,^gv)=0$. 
For $u,v\in (V^h)^\perp$,
both sides vanish
as they are alternating in $u$ and $v$ and $(V^h)^\perp$
has dimension $1$.
Lastly, \cref{JacobiConverse} implies \cref{JacobiGrp}(3) is satisfied.
Therefore, $(\lambda \ot hg) \oplus (\alpha \ot h)$ is a cocycle by \cref{JacobiGrp}.

In fact, we observe that
 $(\lambda \ot hg) \oplus (\alpha \ot h)$ is a 
distinguished coset representative 
as in \cref{UniqueReps}:
$\pi_h\alpha\equiv 0$,
 $\alpha\equiv 0$ on $\Wedge^2 V^h$, and, whenever
 $\chi_h$ is nontrivial,
 $\FF^{\chi_h}=\{0\}$ 
 so
  $\lambda(x)=\lambda'=0$.
It is straightforward
to check that
$\Phi \Phi'=\Phi' \Phi
=1$ using \cref{UniqueReps}.

\vspace{1.5ex}

%%%%%%%%%%%%%%%%%%%
\noindent
{\bf Codimension two contributions:}
Fix $h$ in $G$ with $\codim V^h=2$ and write $V = V^h \oplus (V^h)^\perp$ with fixed basis elements $v_1$ and $v_2$ of $(V^h)^\perp$.
Then $\hat{v}_1=v_1-\, ^h v_1$ and
$\hat{v}_2=v_2-\, ^h v_2$ form a basis for $V_h$.
Define
$$
\Phi: \HH^2_{-1}(h) \longrightarrow  (V/V_h)^{\chi_{h}}
\qquad
\text{by}
\qquad
(\lambda \ot hg) \oplus (\alpha \ot h)
\mapsto 
\alpha(v_1\wedge v_2)+ V_h
\,,
$$
where $(\lambda \ot hg) \oplus (\alpha \ot h)$ is a
distinguished coset representative for $\HH^2_{-1}(h)$ 
as in \cref{UniqueReps}.
We show that $\alpha(v_1\wedge v_2)+ V_h$ is $\chi_h$-invariant and hence $\Phi$ has the indicated codomain.
First observe that
$g$ acts on the $1$-dimensional space $\Wedge^2 (V/V^h)$ by the scalar 
$\chi_h(g)=\det[g]_{V/V^h}$
and 
$\alpha(u\wedge v) = 0$ for
any $u \in V^h$ and $v \in V$,
so
\begin{equation}
\label{Eqn2a}
\alpha(\,^gv_1\wedge \, ^gv_2) =\chi_h(g) \cdot \alpha(v_1\wedge v_2)
\, 
\end{equation}
and
\begin{equation}
    \label{Eqn2b}
\begin{aligned}
\alpha(v_1\wedge v_2)
& -
\chi_h(g) \cdot
\,^{g^{-1}}(
\alpha(v_1\wedge v_2)) + V_h
=
\alpha(v_1\wedge v_2)
-
\,^{g^{-1}}(\chi_h(g) \cdot \alpha(v_1\wedge v_2))
+ V_h
\\
&=
\alpha(v_1\wedge v_2)
-\,^{g^{-1}}(\alpha(\,^gv_1
\wedge \,^gv_2))+V_h
=
(\alpha-\,^{g^{-1}}\alpha)(v_1
\wedge v_2)+V_h
\, .
\end{aligned}
\end{equation}
Then by \cref{JacobiGrp}(2),
$(\alpha-\,^{g^{-1}}\alpha)(v_1
\wedge v_2)+V_h
= 0$,
and the image of $\Phi$ is $\chi_h$-invariant.

Now we construct an inverse 
to $\Phi$,
$$
\Phi': \HH^2_{-1}(h) \longleftarrow \left(V/V_h \right)^{\chi_h},
\qquad
\big(
(\lambda \ot hg) \oplus 
(\alpha \ot h)\big) + \BBB^2_{-1}(h) \leftmapsto
v + V_h
\, ,
$$
as follows
using the projection map
$\pi_h^\perp: V \rightarrow
(V_h)^\perp$.
For any coset $v+V_h$ in $(V/V_h)^{\chi_h}$,
define the map
$\alpha: \Wedge^2 V \to V$ 
by setting 
\begin{equation}
    \label{Eqn6}
\alpha(u\wedge w)=0
\quad\text{for all $u \in V^h$, $w\in V$}
\quad\text{ and }\quad
\alpha(v_1\wedge v_2)=\pi_h^{\perp}(v)
\in (V_h)^\perp
\, .
\end{equation}
Note that this is independent
of choice of coset representative
$v$.
Before defining $\lambda$, we observe that $(\alpha - \,^{g^{-1}}\alpha)(v_1\wedge v_2)$ lies in $V_h$:
\cref{Eqn6} implies that
\cref{Eqn2a,Eqn2b}
hold and thus
$$
\begin{aligned}
(\alpha
- \,^{g^{-1}}\alpha)
(v_1
\wedge v_2) + V_h
&=
\alpha(v_1\wedge v_2)
 -
\chi_h(g)\, 
\,^{g^{-1}}(
\alpha(v_1\wedge v_2)) + V_h
\\
&=
\pi_h^\perp(v)
- \chi_h(g)\,
\,^{g^{-1}}\big(
\pi_h^\perp(v)\big)
+ V_h
= 
0
\, 
\end{aligned}
$$
as
$v+ V_h=\pi_h^\perp(v)+V_h$ is $\chi_h$-invariant
and $G$ preserves $V_h$ set-wise.
Thus we may write
$$
(\alpha - \,^{g^{-1}}\alpha)(v_1\wedge v_2)
=
a \, (v_1-\,^hv_1) - b\, (v_2-\,^hv_2)
\qquad
\text{for some $a,b \in \FF$}
\,.
$$
Define a linear function
$\lambda: V \to \FF$ by setting
$\lambda\equiv 0$ on $V^h$, 
$
\lambda(v_2) = a$,
and $
\lambda(v_1) = b
$.
We argue that $(\lambda \ot hg) \oplus (\alpha \ot h)$ is a cocycle 
in $\ZZZ^{2}_{-1}(h)$
by checking the three conditions of \cref{JacobiGrp}.
Since $\lambda \equiv 0$ on $V^h$ and $\im T \subset V^G \subset V^h$, \cref{JacobiGrp}(1) holds.
We next
verify \cref{JacobiGrp}(2), i.e.,
\begin{equation*}
(\alpha-\, ^{g^{-1}}\alpha)
(u\wedge w)
=
\lambda(w)(u-\, ^{h} u)
-
\lambda(u)(w-\, ^{h} w)
\quad\text{ for all }
u,w \in V
\, .
\end{equation*}
This holds for $u=v_1$ and $w=v_2$ by construction of $\lambda$.
It also holds when $u \in V^h$ and $w \in V$:
The right-hand side 
vanishes since $u-\,^hu=0$ and $\lambda(u)=0$ 
and the left-hand side 
vanishes as well
since 
$\alpha(u\wedge w)=0=
\alpha(\,^gu\wedge  \,^gw)$ 
by the construction of $\alpha$
as both $u$ and $\,^gu$ lie in 
$V^h$.
Thus \cref{JacobiGrp}(2) is satisfied.
\cref{JacobiGrp}(3) is satisfied by \cref{JacobiConverse}.
So $(\lambda \ot hg) \oplus (\alpha \ot h)$ is a cocycle.

Notice that  this cocycle is a distinguished coset representative as in \cref{UniqueReps}.
It is then straightforward to check that $\Phi \Phi' = \Phi' \Phi = 1$
using \cref{UniqueReps}
and \cref{JacobiGrp}.
\end{proof}
%%

%%%%%%%%%%%%%%%%%%%%%%%%%%
%%%%%%%%%%%%%%%%%%%%%%%%%%

%%%%%%%%%%%%%%%%%%%%%%%%%%%
%%%
%%%

\vspace{1ex}

\begin{remark}{\em
\label{RecoverNonModResult}
We recover from the last theorem
the description of 
the graded deformation cohomology 
$\HH^2_{-1}(A)$ for $A=S(V)\rtimes G$
in the nonmodular setting.
Indeed,
for $G$ cyclic 
with $|G|$ and $\cchar \FF$ coprime
and $\FF$ algebraically closed,
\cref{MainThm} 
gives
\cref{NonModularCohCyclic}.
In this setting,
$\im T = V^G$ and $V/V_h\cong V^h$
as an $\FF G$-module.
In addition, if $h$ is a reflection, then $h$ is diagonalizable with $\chi_h(h)\neq 1$
(see \cref{LinearCharacter})
 so
$\FF^{\chi_h}=0$.
Finally, as $h$ acts on $V/V_h$
and on $(V^h)^*$ as the identity
and $\chi_h(h)\neq 1$,
no element of $V/V_h\ot (V^h)^*$
can be $\chi_h$-invariant
so the middle summand of
\cref{MainThm} vanishes.
We give a generalization 
of this phenomenon in the 
next result.
}% end \em
\end{remark}

\vspace{1ex}

\begin{cor}
\label{PartsNonmodular}
Suppose $G \subset \GL(V)$ is a cyclic group acting on  $V = \FF^n$.
Let $|G|=p^k r$ for some
$k$ with $r$ and $p=\cchar \FF$ 
coprime and suppose $\FF$ 
contains a primitive
$r$-th root of unity. Then
the space of infinitesimal
graded deformations of $A=S(V)\rtimes G$ is isomorphic
as an $\FF$-vector space
to
$$
\begin{aligned}
&
\HH^2_{-1}(A)
\ \cong\ 
(V^G/\im T)^*
\oplus
\big( 
V \ot\, \Wedge^2 V^*
\big)^G
\, \oplus\hspace{-2ex}
\bigoplus_{\substack{h \in G \\ \codim V^h =1
\\ \det(h)=1\rule{0ex}{1.5ex}}}
\hspace{-3ex}
\left(\FF \oplus
\big(V/V_h \ot (V^h)^*\big)
\right)^{\chi_h}
\hspace{-1ex}
\ \oplus\hspace{-2ex}
\bigoplus_{\substack{h \in G \\ \codim V^h =2
\\ \rule{0ex}{1.5ex}
\det (h) =1}} 
\hspace{-3ex}
\left(
V/V_h
\right)^{\chi_h}
.
\end{aligned}
$$
\end{cor}
%%%
%%%
%
\begin{proof}
As $\FF$ has a primitive $r$-th
root-of-unity,
we may choose a Jordan canonical
form for the generator $g$ of $G$.
Say
$\codim V^h=1$ for some $h=g^i$ in $G$ with $\det h\neq 1$. Then $h$ is diagonalizable
with a single non-$1$ eigenvalue.
Since a single Jordan block $B$
of $g$
has only one eigenvalue $\xi$ with multiplicity the size of the block
and $B^i$
has only one eigenvalue $\xi^i$
with the same multiplicity,
there must be a $1\times 1$ Jordan
block of $g$ corresponding
to an eigenvector $g$ and $h$ share 
not in $V^h$. Thus $G$ preserves set-wise both $V^h$ and 
the choice of vector space complement
$(V^h)^\perp=V_h=\im(1-h)$.
We
identify $V/V_h$
with $V^h$ 
as an $\FF G$-module 
and observe 
as in \cref{RecoverNonModResult}
that
$\chi_h(h)\neq 1$
yet $h$ fixes 
$\FF \oplus
\big(V/V_h \ot (V^h)^*\big)$
point-wise,
so $0$ is the only $\chi_h$-invariant
element
of this space.
We use a similar argument
when $\codim V^h=2$
with $\det h\neq 1$:
$h$ must have two
non-$1$ eigenvalues 
and thus there must
be one or two
Jordan blocks of $g$
 corresponding
to $V_h$, a vector space complement to $V^h$;
we again identify $V/V_h$ with $V^h$
as an $\FF G$-module
and note that $\chi_h(h)=\det h$
to conclude that the space of
$\chi_h$-invariants is the zero space.
\end{proof}

%%%%%%%%%%%%%%%%%%%%%%%%%%%
%%%%%%%%%%%%%%%%%%%%%%%%%%%
%%%%%%%%%%%%%%%%%%%%%%%%%%%%
\section{Applications
to deformation theory}
\label{ApplicationSection}
Here we demonstrate how to use the description of
infinitesimal graded deformations in \cref{MainThm} to 
obtain explicit graded deformations
in a particular setting.
We use three resolutions of
$A=S(V) \rtimes G$
for a cyclic group $G$
acting on $V \cong \FF^n$
using the 
twisted product
resolutions of \cite{SW19}:
\begin{itemize}
\item
Resolution $X_{\DOT}$
is a twisting of a
periodic resolution for $\FF G$ with the Koszul resolution for $S(V)$ (see \cref{PeriodicResolutionSection}),
\item
Resolution $Y_{\DOT}$ is a twisting of
the bar resolution for $\FF G$ and the Koszul resolution for $S(V)$ (see \cite{SW15,SW19} for the differentials),
\item
Resolution $Z_{\DOT}$
is the bar resolution of $A$.
\end{itemize}
Each is an $A^e$-free resolution of $A$, 
and we make a choice of 
chain maps between the three resolutions:
\[
\begin{tikzcd}
X_{\DOT}: 
&[-5ex] 
\cdots 
\to 
&[-6ex] 
\arrow[d,  
shift right = .75ex
] 
X_3 
\arrow[r]
&[3ex]
\arrow[d, 
shift right =2.5ex] 
X_2 
\to 
&[-7ex] 
\cdots 
&[0ex] 
\text{(for finding cohomology representatives)}
\\
Y_{\DOT}: 
&[-5ex]
\cdots 
\to 
&[-6ex] 
\arrow[u, swap, 
shift right=.75ex] 
\arrow[d, swap, 
shift right = .75ex] 
Y_3 
\arrow[r] 
& 
\arrow[u, swap, 
shift right=-1ex] 
\arrow[d, swap, 
shift right = 2.5ex] 
Y_2  \to &[-7ex] 
\cdots 
&[0ex] 
\text{(for computing Gerstenhaber brackets)}
\\
Z_{\DOT}: 
&[-5ex]
\cdots 
\to &[-6ex] 
\arrow[u, swap,  
shift right =.75ex] 
Z_3 
\arrow[r] 
& 
\arrow[u, swap,  
shift right=-1ex] 
Z_2  
\to 
&[-7ex] 
\cdots
&[3ex] 
\hspace{0ex}
\text{(for describing deformations)}
\end{tikzcd}\]
Resolutions $Y_{\DOT}$ and $Z_{\DOT}$ here may be used for any finite group $G$. 
Resolution $X_{\DOT}$ is reserved for 
a finite cyclic group $G$.

\subsection*{Cyclic
transvection groups}
We now turn to the unipotent cyclic
groups acting on $V=\FF^2$, i.e.,
the cyclic transvection groups.
Consider $G=\langle g \rangle \subset\GL(V)$ 
with $\characteristic \FF=p>0$ 
and $g$ a nondiagonalizable
reflection.
Then $|G|=p$ 
and there is a basis $v_1, v_2$
of $V$ so that
$$
g(v_1) = v_1
\quad\text{ and }\quad
g(v_2)=v_1+v_2,
\qquad
\text{ i.e., }\quad
g 
= 
\left(
\begin{smallmatrix}
1 & 1 \\
0 & 1
\end{smallmatrix}
\right)
\quad\text{ and }\quad
G=\big\{ 
\left(
\begin{smallmatrix}
1 & * \\
0 & 1
\end{smallmatrix}
\right)
\big\}
\subset \GL_n(\FF)
\, .
$$
Here  $S(V)=
\mathbb{F}[v_1,v_2]$ and 
$\im T=\{0\}$ by \cref{ReflectionLemma} as $g$ itself is a nondiagonalizable reflection.
For each $h$ in $G$,
$\codim V^h\leq 1$ and
$
V^h = V_h=
V^G=\FF v_1$
whereas
$V/V_h \cong \FF v_2$.
Note that
$\chi_h:G\rightarrow \FF^{\times}$ is the trivial character
in this setting
as each $h'$ in $G$
acts trivially on each $V/V_h$. 
For $A=S(V)\rtimes G$, 
\cref{MainThm} thus implies that
\begin{equation}
\begin{aligned}
\label{TransvectionCoh}
\HH^2_{-1}(A) 
&\ \cong\ 
\underbrace{
(V^G)^* \oplus 
(V \ot \Wedge^2 V^*)^G }_{\text{contribution of $1_G$}}
\ \oplus\ 
\bigoplus_{
\substack{h \in G\ \\ h \neq 1_G}}
\underbrace{
\big( \FF \oplus
(V/V^G \ot (V^G)^*)
\big)^{G}}_{\text{\ \ 
contribution of reflections}}
\\
&\ =\ 
(\FF v_1)^* \oplus 
(\FF v_1 \ot \FF v_1\wedge v_2)
\ \oplus\ 
\bigoplus_{
\substack{h \in G\ \\ h \neq 1_G}}
\big( \FF \oplus
(V/\FF v_1 \ot (\FF v_1)^*)
\big)
\ \cong\ \FF^{2p}
\, .
\end{aligned}
\end{equation}

%%%%%%%%%%%%%%%%%
\subsection*{Lifting a Hochschild cocycle
to a deformation}
Consider $1 + (v_2+\FF v_1) \ot v_1^*$
in the summand
$\FF \oplus 
(V/\FF v_1 \ot (\FF v_1)^*)$ of
 $\HH^2_{-1}(A)$ in \cref{TransvectionCoh} corresponding to 
$h=g$ where $v_1^*$ is the map $v_1\mapsto 1$.
Using the isomorphism in the proof
of \cref{MainThm},
we identify this element with the cocycle $\gamma_{_X}^{\vphantom{X}}$ on $X_{\DOT}$ in $\CCC^2_{-1}(A)=
(V^* \ot \FF G)
\oplus
(V \ot \Wedge^2 V^* \ot \FF G )$ given by 
$$
\gamma_{_X}^{\vphantom{X}}(v_1) = g^2,
\qquad
\gamma_{_X}^{\vphantom{X}}(v_2)=g^2,
\qquad
\text{and}
\qquad
\gamma_{_X}^{\vphantom{X}}(v_2 \wedge v_1) 
= v_2\ot g
\,.
$$
We lift $\gamma_{_X}$
to a specific
$2$-cocycle
$\gamma=\gamma_{_Y}$
on the bar-twisted-Koszul resolution,
$$\gamma: Y_2=(\FF G \ot \FF G)
\oplus (\FF G \ot V)
\oplus (\Wedge^2 V)
\rightarrow S(V)\ot \FF G,
$$
by applying a choice of
chain map 
$X_{\DOT}\rightarrow Y_{\DOT}$:
For $0\leq i,j<|G|$,
calculations give
\begin{equation}
\label{Deformation}
\begin{aligned}
\gamma(g^i \ot g^j) = 0\, ,
\ \
\gamma(g^i \ot v_1)
=i \, g^{i+1}
\,,
\ \
\gamma(g^i \ot v_2) 
= 
\tbinom{i+1}{2}\, g^{i+1}
\,, 
\ \
\text{ and }
\gamma(v_2 \wedge v_1)
=
v_2\ot g
\end{aligned}
\, .
\end{equation}

We argue that $\gamma$ lifts to a graded deformation of 
$\FF[v_1,v_2] \rtimes G$.
We consider the lifting conditions
for $\gamma$ on $Y_{\DOT}$ (\cite[Theorem 5.3]{SW18}):
As $\dim V=2$, $\gamma$
lifts to a deformation
if and only if
$$
\text{}
\ \ 
\text{$[\gamma, \gamma]=0$ as a cochain on $Y_3$.}
$$
Here, $[\ ,\ ]$
is the Gerstenhaber bracket
on Hochschild cohomology
lifted to the resolution $Y_{\DOT}$.
Since $\gamma$ has degree $-1$,
the square bracket
$[\gamma, \gamma]$ 
is a $3$-cochain
of degree $-2$. Then as $\dim V=2$,  we need only check the value of the square bracket on $h \ot v_1 \wedge v_2$ in $\FF G \ot \Wedge^2 V$
for $h$ in $G$.
We use the formula on the right-hand side of \cite[Theorem 6.1(2)]{SW18}:
$$
\begin{aligned}
[\gamma,\gamma](g^i \ot v_1 \wedge v_2)
&
=
\gamma(\gamma(g^i \ot v_2) \ot v_1)
-\gamma(\gamma(g^i \ot v_1) \ot v_2)
+\gamma(g^i \ot  v_2)\, g
\\
&=
\tbinom{i+1}{2} \, \gamma(g^{i+1} \ot v_1)
-i\, \gamma(g^{i+1} \ot v_2)
+ \tbinom{i+1}{2}\, g^{i+2}
\\
&
=
\tbinom{i+1}{2}(i+1)\, g^{i+2}
-i \tbinom{i+2}{2} g^{i+2}
+
\tbinom{i+1}{2} \, g^{i+2}
\,.
\end{aligned}
$$
A computation shows that the right-hand side is zero and thus $[\gamma,\gamma]=0$ as a cochain. 
Hence $\gamma$
in $\HH^2(A)$
is not just an infinitesimal
graded deformation, but the infinitesimal
(first multiplication map)
of a (formal) graded deformation
of $A=S(V)\rtimes G$.

%%%%%%%%%%%%%%%%
\subsection*{Graded deformation
as an explicit Drinfeld orbifold algebra}
We now give that graded deformation
explicitly.
For a pair of
linear parameter functions
$$
\lambda: \FF G \ot V \to \FF G
\qquad
\text{and}
\qquad
\kappa: \Wedge^2 V \to V \ot \FF G
\,,
$$
let $\cH_{\lambda,\kappa}$ be
the $\FF$-algebra 
(see \cite{SW18}, for example)
generated by $\FF G$ and $V$ with defining relations 
$$
h u - \,^huh = \lambda(h \ot u) 
\quad
\text{and}
\quad
uv-vu = \kappa(u \wedge v)
\quad
\text{for all $h \in G$ and $u,v \in V$}
\,.
$$
For $\gamma$ defined in \cref{Deformation}, 
now consider the 
algebra $\cH_{\lambda, \alpha}$ defined with specific parameters
$$
\lambda=\gamma\big\vert_{\FF G \ot V} 
\quad
\text{ and } 
\quad 
\alpha=\gamma
\big\vert_{  \scalebox{.75}{$\Wedge^2 V$} }\, ,
\qquad\text{ so }
\gamma=\lambda \oplus \alpha
\, .
$$ 
We identify
$\alpha$ with $\kappa^L$ and set $\kappa^C$ to zero in \cite[Theorem 6.1]{SW18}
and check the six conditions of that theorem  to
conclude that $\cH_{\lambda,\alpha}$
satisfies the PBW property
and is isomorphic to
$\FF[v_1,v_2] \rtimes G$
as a vector space.
Thus the 
$\FF$-algebra
$
\cH_{\lambda,\alpha} 
$ generated
by $\FF G$ and $V$
with relations
$$gv_1-v_1g=g^2,\quad gv_2-v_1g-v_2g=
0,\quad
\text{ and }\quad
v_2v_1-v_1v_2=v_2g 
$$
is a PBW deformation of $\FF[v_1,v_2] \rtimes G$.
This analog of a universal enveloping algebra is called a {\em Drinfeld orbifold algebra}.
%
%

%%%%%%%%%%%%%%%%%%%%%%%%%%%
%%%%%%%%%%%%%%%%%%%%%%%%%%%
%%%%%%%%%%%%%%%%%%%%%%%%%%%
%%%%%%%%%%%%%%%%%%%%%%%%%%%
%%%%%%%%%%%%%%%%%%%%%%%%%%%%
%%%%%%%%%%%%%%%%%%%%%%%%%%%
%%%%%%%%%%%%%%%%%%%%%%%%%%%
%%%%%%%%%%%%%%%%%%%%%%%%%%%
%%%%%%%%%%%%%%%%%%%%%%%%%%%
%%%%%%%%%%%%%%%%%%%%%%%%%%%%
%%%%%%%%%%%%%%%%%%%%%%%%%%%
%%%%%%%%%%%%%%%%%%%%%%%%%%%
%%%%%%%%%%%%%%%%%%%%%%%%%%%
%%%%%%%%%%%%%%%%%%%%%%%%%%%
%%%%%%%%%%%%%%%%%%%%%%%%%%%%
%%%%%%%%%%%%%%%%%%%%%%%%%%%
%%%%%%%%%%%%%%%%%%%%%%%%%%%
%%%%%%%%%%%%%%%%%%%%%%%%%%%
%%%%%%%%%%%%%%%%%%%%%%%%%%%
%%%%%%%%%%%%%%%%%%%%%%%%%%%%
%%%%%%%%%%%%%%%%%%%%%%%%%%%
%%%%%%%%%%%%%%%%%%%%%%%%%%%

\end{document}